\documentclass[10pt]{article}
        
\usepackage{amssymb,amsmath,amsfonts,amsthm}
\usepackage{graphicx}
\usepackage[a4paper,includeheadfoot,left=4cm,right=4cm,top=3cm,bottom=3cm]{geometry}
\usepackage[stretch=10]{microtype} 

\usepackage[font={footnotesize}]{caption}

\usepackage[numbers,comma,sort&compress]{natbib}
\makeatletter\let\@fnsymbol\@arabic\makeatother

\usepackage{color}
\definecolor{marin}{rgb}{0.,0.3,0.7}
\usepackage[colorlinks,citecolor=marin,linkcolor=marin,urlcolor=marin,bookmarksopen,bookmarksnumbered]{hyperref}

\newcommand{\al}{\alpha}
\newcommand{\de}{\delta}

\newcommand{\ka}{\kappa}
\newcommand{\si}{\sigma}
\newcommand{\ph}{\varphi}
\newcommand{\Om}{\Omega}
\newcommand{\Z}{\mathbb{Z}}

\newcommand{\R}{\mathbb{R}}
\newcommand{\C}{\mathbb{C}}
\newcommand{\Ec}{\mathcal{E}}
\newcommand{\Pc}{\mathcal{P}}

\newcommand{\Rc}{\mathcal{R}}
\newcommand{\Uc}{\mathcal{U}}
\providecommand{\abs}[1]{\lvert#1\rvert}
\providecommand{\absbig}[1]{\bigl\lvert#1\bigr\rvert}
\providecommand{\absBig}[1]{\Bigl\lvert#1\Bigr\rvert}
\providecommand{\absbigg}[1]{\biggl\lvert#1\biggr\rvert}

\providecommand{\norm}[1]{\lVert#1\rVert}
\providecommand{\normbig}[1]{\bigl\lVert#1\bigr\rVert}

\providecommand{\normbigg}[1]{\biggl\lVert#1\biggr\rVert}
\providecommand{\normv}[1]{\ensuremath{{\lVert\hskip-1pt\lvert}#1{\rvert\hskip-1pt\rVert}}}
\providecommand{\normvbig}[1]{\ensuremath{{\bigl\lVert\hskip-1pt\bigl\lvert}#1{\bigr\rvert\hskip-1pt\bigr\rVert}}}

\providecommand{\normvbigg}[1]{\ensuremath{{\biggl\lVert\hskip-1pt\biggl\lvert}#1{\biggr\rvert\hskip-1pt\biggr\rVert}}}
\providecommand{\skla}[1]{\langle#1\rangle}
\providecommand{\sklabig}[1]{\bigl\langle#1\bigr\rangle}

\providecommand{\kla}[1]{(#1)}
\providecommand{\klabig}[1]{\bigl(#1\bigr)}
\providecommand{\klaBig}[1]{\Bigl(#1\Bigr)}
\providecommand{\klabigg}[1]{\biggl(#1\biggr)}

\newcommand{\formulatext}[1]{\qquad\text{#1}\qquad}
\newcommand\myfor{\formulatext{for}}
\newcommand\myforall{\formulatext{for all}}
\newcommand\myand{\formulatext{and}}

\newcommand\mywith{\formulatext{with}}
\newcommand{\iu}{\mathrm{i}}
\newcommand{\e}{\mathrm{e}}
\newcommand{\drm}{\mathrm{d}}

\DeclareMathOperator{\sinc}{sinc}

\DeclareMathOperator{\Op}{Op}
\newtheorem{theorem}{Theorem}[section]
\newtheorem{lemma}[theorem]{Lemma}
\newtheorem{proposition}[theorem]{Proposition}
\theoremstyle{definition}
\newtheorem{assum}{Assumption}
\newtheorem{remark}[theorem]{Remark}
\newtheorem{example}[theorem]{Example}

\definecolor{mygreen}{rgb}{0.,0.7,0.3}
\definecolor{myorange}{rgb}{0.9,0.5,0.0}
\definecolor{myred}{rgb}{0.9,0.2,0.2}

\title{Trigonometric integrators\\ for quasilinear wave equations}

\author{Ludwig Gauckler\,\thanks{Institut f\"ur Mathematik,
          Freie Universit\"at Berlin,
          Arnimallee 9,
          D-14195 Berlin, Germany
          ({\tt gauckler@math.fu-berlin.de}).}
         \and
         Jianfeng Lu\,\thanks{Departments of Mathematics, Physics, and
           Chemistry,
           Duke University, 
           Box 90320, 
           Durham, NC 27708, USA
           ({\tt jianfeng@math.duke.edu}).}
         \and
         Jeremy L.\ Marzuola\,\thanks{Department of Mathematics,
           UNC-Chapel Hill, 
           CB\#3250 Phillips Hall, 
           Chapel Hill, NC 27599, USA
           ({\tt marzuola@math.unc.edu}).}
         \and
         Fr\'ed\'eric Rousset\,\thanks{Laboratoire de Math\'ematiques d'Orsay (UMR 8628), Universit\'e Paris-Sud, 91405 Orsay Cedex, France  and Institut Universitaire de France 
          ({\tt frederic.rousset@math.u-psud.fr}).}
          \and
         Katharina Schratz\,\thanks{Fakult\"at f\"ur Mathematik, 
           Karlsruhe Institute of Technology, 
           Englerstr.\ 2, 
           D-76131 Karlsruhe, Germany
           ({\tt katharina.schratz@kit.edu}).}
}

\begin{document}

\maketitle

\begin{abstract}
Trigonometric time integrators are introduced as a class of explicit numerical methods for  quasilinear wave equations. Second-order convergence for the semi-discretization in time with these integrators is shown for a sufficiently regular exact solution. The time integrators are also combined with a Fourier spectral method into a fully discrete scheme, for which error bounds are provided without requiring any CFL-type coupling of the discretization parameters. The proofs of the error bounds are based on energy techniques and on  the semiclassical G\aa rding inequality. \\[1.5ex]
\end{abstract}

\section{Introduction}

The topic of the present paper is the numerical analysis of \emph{quasilinear wave equations}. Such wave equations show up in a variety of applications, ranging from elastodynamics to general relativity.
While the (local-in-time) analysis of them is well-developed since the seventies, the papers \cite{Kato1975} by Kato and \cite{Hughes1976} by Hughes, Kato \& Marsden being major contributions to the local well-posedness theory, and has meanwhile found its way into classical monographs on partial differential equation, see, for instance, the monograph \cite{Taylor1991} by Taylor, as well as the books by Sogge \cite{sogge1995lectures} and H\"ormander \cite{lars1997lectures}, the numerical analysis of quasilinear wave equations is much less developed. The main challenge is, of course, the numerical treatment of the quasilinear term in the equation.

In the present paper, we focus on quasilinear wave equations of the form
\begin{equation}\label{eq-qlw}
\partial_{t}^2 u = \partial_{x}^2 u - u + \ka \,a(u) \, \partial_{x}^2 u + \ka \, g\bigl(u,\partial_x u\bigr)
\end{equation}
where  $g$ and $a$ are  smooth and real-valued functions such that $g(0, 0)=a(0)=0$. We consider real-valued solutions to \eqref{eq-qlw} with $2\pi$-periodic boundary conditions in one space dimension, $x\in\mathbb{T}=\R/(2\pi\Z)$, for initial values
\begin{equation}\label{eq-qlw-init}
u(\cdot,0) = u_0, \qquad \partial_tu(\cdot,0) = \dot{u}_0
\end{equation}
given at time $t=0$.  The  real-valued  parameter $\kappa$ will be used to emphasize the strength of the nonlinearities,  and 
we will be interested both in the regime  where $\kappa$ is small so that the nonlinearities are small
 and the regime where $\kappa$ is of order one. Quasilinear wave equations of this form with small $\abs{\ka}$ have been extensively studied by Groves \& Schneider \cite{Groves2005}, Chong \& Schneider \cite{Chong2013}, Chirilus-Bruckner, D\"ull \& Schneider \cite{Chirilus-Bruckner2015} and D\"ull \cite{Duell2016}:  the equations from the class \eqref{eq-qlw} are prototypes for  models in nonlinear optics with a nonlinear Schr\"odinger equation as a modulation equation \cite[Introduction]{Groves2005}.  Many examples from elasticity and fluid mechanics can also be reduced to quasilinear wave
equations under the form \eqref{eq-qlw}.   Relevant applications from elasticity and general relativity appeared in \cite{Hughes1976} for instance, though of course many of the most physically interesting examples occur in dimensions $2$ and include potential dependence upon $\partial_x u$ in $a$ and possible dependence upon $\partial_t u$.  Using the techniques developed here to study smooth solutions in relevant higher dimensional models will be a topic for future work and would likely require some higher regularity assumptions on the solutions. 

\medskip

The principal difficulty in the numerical discretization of \eqref{eq-qlw} is the quasilinear term $\ka a(u) \partial_{x}^2 u$. For typical explicit methods for the discretization in time, in which the numerical approximation at a discrete time depends in an explicit way on this quasilinear term at some previous time step, there is a risk of \emph{losing derivatives}, in the sense that a control of a certain number of spatial derivatives of the numerical solution requires the control of \emph{more} derivatives of the numerical solution at previous time steps. This phenomenon is by far not restricted to quasilinear wave equations, and an established way to prevent a loss of derivatives is to resort to carefully chosen implicit methods. 
In the case of quasilinear wave equations, this route has been taken recently by Hochbruck \& Pa\v{z}ur \cite{Hochbruck2015} and Kov\'{a}cs \& Lubich \cite{Kovacs}, who propose and study implicit and semi-implicit methods of Runge--Kutta type for semi-discretization in time 
for a more general class of quasilinear evolution  equations.  

In the present paper, we take another route and show how a class of \emph{explicit} time discretizations can be used to numerically solve the quasilinear wave equation \eqref{eq-qlw} (and also how it can be combined with a Fourier spectral method in space). The considered class of methods is the class of \emph{trigonometric integrators}, which is described in detail in Section \ref{sec-method}. These methods are exponential integrators and have been originally developed for highly oscillatory ordinary differential equations, see, for instance, \cite{GarciaArchilla1999} or Chapter XIII of the monograph \cite{Hairer2006}. Meanwhile, they were recognized to work well also for wave equations in the semilinear case, see \cite{Bao2012,Cano2010a,Cano2013,Cohen2008a,Dong2014,Gauckler2015,Gauckler}. We show here, how these explicit methods can be put to use also in the quasilinear case. A careful choice of the filters in these integrators turns out to play a crucial role in avoiding the above-mentioned loss of derivatives. 

For the considered and derived trigonometric integrators, we rigorously prove second-order convergence in time. We also prove convergence of a fully discrete method which is based on a combination with a spectral discretization in space, without requiring any CFL-type coupling of the discretization parameters. See Section \ref{sec-results} for statements of the error bounds. 
The proofs of our convergence results are based on \emph{energy techniques} as they are widely used in mathematical analysis to prove well-posedness of quasilinear equations and also well-established in the numerical analysis of quasilinear \emph{parabolic} equations, see, e.g., \cite{Lubich1995}. Furthermore, they have been applied in the recent analysis of \emph{implicit} methods for quasilinear hyperbolic equations in \cite{Hochbruck2015} for (semi-)implicit Euler methods and in \cite{Kovacs} for (linearly) implicit midpoint methods and implicit Runge--Kutta methods.

Here, we are interested in the analysis of \emph{explicit} trigonometric integrators for quasilinear wave equations in the form
\eqref{eq-qlw}. Note that in contrast to the semilinear case studied in \cite{Bao2012,Cano2010a,Cano2013,Cohen2008a,Dong2014,Gauckler2015,Gauckler},  the proof uses   energy techniques with a nontrivial modified discrete energy to prove stability of the methods under appropriate assumptions on the filters. In addition, in the case of non-small $\ka$, we also need tools from semiclassical pseudodifferential calculus to ensure that   the  modified discrete energy is positive.

We mention that recently in \cite{gonzalez2015higher,gonzalez2016higher} explicit exponential integrators for quasilinear \emph{parabolic} problems in Banach spaces were considered.  Analysis of quasilinear parabolic equations can generally be done using simpler techniques stemming from the regularization implicit in the diffusion operators, but quasilinear waves must be handled with more care given the lack of smoothing properties of the leading order operator. We also point out that the exponential integrators in \cite{gonzalez2015higher,gonzalez2016higher} are based on solving exactly a differential equation with the linearization of the quasilinear part on the right, whereas the trigonometric integrators considered in the present paper are based on solving exactly a differential equation with the pure linear part $\partial_x^2u-u$ on the right, which is usually simpler from a computational point of view.

\medskip

The methods and their analysis as presented in the present paper can be extended to higher spatial dimensions or to quasilinear wave equations \eqref{eq-qlw} without Klein--Gordon term $-u$ on the right-hand side.
Moreover, we could also only assume that $a$ and $g$ are smooth  on an open subset and deal with smooth solutions
that stay in this subset on  the considered time interval.  In this way,  our scheme can be used to approximate  the classical $p$-system of elasticity and gas dynamics (as long as the solutions are smooth and with no vacuum).  See \cite{LPT} for instance for a discussion of this model.  It would be interesting to see if  our methods  and their analysis  can be extended to quasilinear wave equations \eqref{eq-qlw} with 
  a  semilinear  term $g$ that depends also on $\partial_{t} u$ (for example in order  to handle the  equations considered in \cite{Doerfler2016}), to    the more abstract  classes of equations considered in \cite{Hochbruck2015,Kovacs}, or to other equations with a possible loss of derivatives in explicit numerical methods, such as quasilinear Schr\"odinger equations as considered in \cite{Lu2015}.
  
\medskip

The article is organized as follows. In Section \ref{sec-method}, we introduce the considered trigonometric integrators, for which we state global error bounds in Section \ref{sec-results}. The proof of the error bound for the semi-discretization in time is given in Section \ref{sec-error}, and the one for the full discretization in Section \ref{sec-error-full}. The necessary tools from semiclassical pseudodifferential calculus are collected in the appendix.

\medskip

\noindent\textbf{Notation.} By $H^s=H^s(\mathbb{T})$, $s\ge 0$, we denote the usual Sobolev space, equipped with the norm $\norm{\cdot}_s$ given by
\[
\norm{v}_s^2 = \sum_{j\in\Z} \skla{j}^{2s} \abs{\hat{v}_j}^2 \myfor  v(x)=\sum_{j\in\Z} \hat{v}_j \e^{\iu jx}
\]
with the weights
\[
\skla{j} = \sqrt{j^2+1}, \qquad j\in\Z. 
\]
By $\skla{\cdot,\cdot}_s$, we denote the corresponding scalar product, 
\[
\skla{v,w}_s = \sum_{j\in\Z} \skla{j}^{2s} \overline{\hat{v}}_j \hat{w}_j \myfor v(x)=\sum_{j\in\Z} \hat{v}_j \e^{\iu jx}, \quad w(x)=\sum_{j\in\Z} \hat{w}_j \e^{\iu jx}.
\]
We study solutions $(u(\cdot,t),\partial_t u(\cdot,t))$ of the quasilinear wave equation \eqref{eq-qlw} in product spaces $H^{s+1}\times H^s$, on which we use the norm
\[
\normv{(u,\dot{u})}_s = \klabig{\norm{u}_{s+1}^2+\norm{\dot{u}}_s^2}^{1/2}. 
\]
For $s>\frac12$, we will make frequent use of the classical  estimates in Sobolev spaces
\begin{subequations}\label{eq-algebra}
\begin{equation}\label{eq-algebra1}
\norm{uv}_0 \le C_s \norm{u}_0 \norm{v}_s, \qquad  \norm{u v}_s \le C_s \norm{u}_s \norm{v}_s
 \end{equation}
 and
\begin{equation}\label{eq-algebra-plus}
\|G(u)\|_{s} \leq \Lambda_{s}(\|u\|_{s}) \|u\|_{s}, \qquad \|G(u)-G(v)\|_{s} \leq \Lambda_{s}(\|u\|_{s}+\norm{v}_s) \|u-v\|_{s},
\end{equation}
\end{subequations}
where  $G$  is any smooth function such that  $G(0)=0$ and $\Lambda_{s}(\cdot)$  is a continuous
 non-decreasing function, see for instance \cite{Taylor11}, Chapter 13, Section 3.

\section{Discretization of quasilinear wave equations}\label{sec-method}

\subsection{Trigonometric integrators for the discretization in time}\label{subsec-trigo}

The quasilinear wave equation \eqref{eq-qlw} can be written in compact form as
\begin{equation}\label{eq-qlw-compact}
\partial_{t}^2 u = - \Om^2 u + \ka f(u).
\end{equation}
with the nonlinearity
\begin{equation}\label{eq-f}
f(u) = a(u)\partial_{x}^2 u + g\bigl(u,\partial_x u\bigr)
\end{equation}
and the linear operator
\[
\Om = \sqrt{-\partial_{x}^2+1},
\]
that is, the operator $\Om$ acts on a function by multiplication of the $j$th Fourier coefficient with $\sqrt{j^2+1}$. 

For the discretization in time of \eqref{eq-qlw-compact}, we use \emph{trigonometric integrators}, see \cite[Section XIII.2.2]{Hairer2006}. We introduce them here as splitting integrators as in \cite{GarciaArchilla1999}, since this interpretation is convenient for the error analysis to be presented in this paper. Written in first-order form $\partial_t (u,\dot{u}) = (\dot{u}, -\Om^2u + \ka f(u))$, equation \eqref{eq-qlw-compact} is split into
\[
\partial_t \begin{pmatrix}u\\ \dot{u}\end{pmatrix} = \begin{pmatrix}\dot{u}\\ -\Om^2u\end{pmatrix} \qquad\text{and}\qquad \partial_t \begin{pmatrix}u\\ \dot{u}\end{pmatrix} = \begin{pmatrix}0\\ \ka f(u)\end{pmatrix}, 
\]
and the usual Strang splitting is applied with time step-size $\tau$. In addition, the nonlinearity $f(u)$ of \eqref{eq-f} is replaced by 
\begin{equation}\label{eq-fhat}
\widehat{f}(u) = \Psi_1 f(\Phi u),
\end{equation}
where
\[
\Psi_1=\psi_1(\tau\Om) \myand \Phi=\phi(\tau\Om)
\]
are filter operators computed from suitably chosen filter functions $\psi_1$ and $\phi$.   Throughout, we assume that the filter functions are bounded and continuously differentiable with bounded derivative.   Denoting by $u_n$ and $\dot{u}_n$ approximations to $u(\cdot,t_n)$ and $\partial_t u(\cdot,t_n)$ at time $t_n=n\tau$, the numerical method thus reads
\begin{equation}\label{eq-method}\begin{split}
\dot{u}_{n}^+ &= \dot{u}_{n} + \tfrac12 \tau \ka \widehat{f}(u_n),\\
\begin{pmatrix}\Om u_{n+1}\\ \dot{u}_{n+1}^-\end{pmatrix} &= \begin{pmatrix}\cos(\tau\Om) & \sin(\tau\Om)\\ -\sin(\tau\Om) & \cos(\tau\Om)\end{pmatrix} \, \begin{pmatrix}\Om u_{n}\\ \dot{u}_{n}^+\end{pmatrix},\\
\dot{u}_{n+1} &= \dot{u}_{n+1}^- + \tfrac12 \tau \ka \widehat{f}(u_{n+1}).
\end{split}\end{equation}
By eliminating the intermediate values $\dot{u}_{n}^+$ and $\dot{u}_{n+1}^-$, one time step of the method is seen to be given by
\begin{equation}\label{eq-method-onestep}\begin{split}
u_{n+1} &= \cos(\tau\Om) u_n + \tau \sinc(\tau\Om) \dot{u}_n + \tfrac12 \tau^2 \sinc(\tau\Om) \ka \widehat{f}(u_n),\\
\dot{u}_{n+1} &= -\Om \sin(\tau\Om) u_n + \cos(\tau\Om) \dot{u}_n + \tfrac12 \tau \cos(\tau\Om) \ka \widehat{f}(u_n) + \tfrac12 \tau \ka \widehat{f}(u_{n+1}).
\end{split}\end{equation}
The numerical flow of this method is denoted in the following by $\ph_\tau$, i.e.,
\begin{equation}\label{eq-numflow}
(u_{n+1},\dot{u}_{n+1})=\ph_\tau(u_{n},\dot{u}_{n}).
\end{equation}

\subsection{On the filter functions}

We collect some assumptions on the filter functions $\phi$ and $\psi_1$ that are going to play an important role in the following. 

\begin{assum}
\label{assum_slw_psibds}
Already in the semilinear case, the bounds
\begin{subequations}\label{eq-assumfilter-bounds}
\begin{align}
\label{propdephi1}&\abs{\phi(\xi)}\le   1  & & \text{and} & & \abs{1-\phi(\xi)}\le c_0\xi^2 &&\text{for all} &&\xi\ge 0,\\
&\abs{\psi_1(\xi)}\le   1  & & \text{and} & & \abs{1-\psi_1(\xi)}\le c_0\xi^2 &&\text{for all} &&\xi\ge 0
\end{align}
\end{subequations}
with a constant $c_0$ are needed for finite-time error bounds, see \cite{Gauckler2015}. 
\end{assum}

\begin{assum}
\label{assum_qlw_psi}
In the quasilinear case, we need in addition that the functions $\phi$ and $\psi_1$ are continuous in $\xi$ and satisfy 
\begin{equation}\label{eq-assumfilter-psi1phi}
\psi_1(\xi) = \sinc(\xi) \phi(\xi) \myforall \xi\ge 0.
\end{equation}
\end{assum}

The condition in Assumption \ref{assum_qlw_psi} has been originally derived in a study on energy conservation properties of trigonometric integrators applied to linear oscillatory ordinary differential equations, see \cite[Equation (2.12)]{Hairer2000}. Surprisingly, it shows up here again in the somehow unrelated context of finite-time error bounds for quasilinear wave equations. 

\begin{assum}
\label{assum_lkap}
We finally need in addition to \eqref{eq-assumfilter-bounds} and \eqref{eq-assumfilter-psi1phi} that,   for prescribed $0<\de<1$ and $A_0\ge 0$ related to the size of $\kappa$ and $a$ in \eqref{eq-qlw} and the solution $u$ to \eqref{eq-qlw}, 
\begin{equation}\label{eq-assumfilter-largedata}
A_0 \sin(\tfrac12\xi)^2 \phi(\xi)^2 \le 1 - \de \myforall \xi\ge 0.
\end{equation}
\end{assum}

\begin{remark}[Small nonlinearity]
\label{rem3pt1}
  The parameters $\de$ and $A_0$ in Assumption \ref{assum_lkap} will be chosen later such that $1+\kappa\, a(u(x,t))\ge \de$ and $\kappa\, a(u(x,t))\le A_0$ for all $x\in \mathbb{T}$ and all $t$ from the time interval under consideration. In particular, $\de$ is small. In the regime $\abs{\kappa}\ll 1$ of a small nonlinearity in \eqref{eq-qlw}, also the value $A_0$ is small, and hence Assumption \ref{assum_lkap} is satisfied in this case for \emph{all} bounded filter functions $\phi$. 
The remaining Assumptions \ref{assum_slw_psibds} and \ref{assum_qlw_psi} are thus sufficient to prove error bounds for small $\abs{\kappa}$ in \eqref{eq-qlw}.   They are, for example, satisfied (with $c_0=1$) for the trigonometric integrator of Hairer \& Lubich \cite{Hairer2000}, where
\begin{equation}\label{eq-trigoHL}
\phi(\xi)=1, \qquad \psi_1(\xi)=\sinc(\xi),
\end{equation}
and the one of Grimm \& Hochbruck \cite{Grimm2006}, where
\begin{equation}\label{eq-trigoGH}
\phi(\xi)=\sinc(\xi), \qquad \psi_1(\xi)=\sinc(\xi)^2. 
\end{equation}
\end{remark}

\begin{remark}[Non-small nonlinearity]
\label{rem3pt2}
  For non-small $\abs{\kappa}$ in \eqref{eq-qlw}, the coefficient $A_0$ in \eqref{eq-assumfilter-largedata} is not small and \eqref{eq-assumfilter-largedata} not always true.  
A new method that we propose here for this case is the trigonometric integrator \eqref{eq-method} with filter functions
\begin{equation}\label{filt}
\phi(\xi) = \sinc(c\xi), \qquad \psi_1(\xi) = \sinc(\xi) \sinc(c\xi) 
\end{equation}
  with
\[
c \geq \frac12 \sqrt{\frac{A_0}{1-\de}}.
\]
Here, $0<\de<1$ and $A_0\ge 0$ are the numbers of Assumption \ref{assum_lkap}.   
This choice of filter functions satisfies Assumptions \ref{assum_slw_psibds} and \ref{assum_qlw_psi} (with $c_0=\max(1,(c^2+1)/6)$), but it also satisfies Assumption \ref{assum_lkap}.   The latter follows from
\[
A_0 \sin(\tfrac12\xi)^2 \phi(\xi)^2 = \frac{A_0}{4c^2} \sinc(\tfrac12\xi)^2 \sin(c\xi)^2 \le \frac{A_0}{4c^2} \le 1-\de.
\]
  Note that a filter function $\sinc(c\xi)$ can be motivated as an averaging of fast forces over a time interval of length $c\tau$, see \cite{GarciaArchilla1999} and \cite[Section XIII.1.4]{Hairer2006}. For $c=1$, the new method \eqref{filt} reduces to the method \eqref{eq-trigoGH} of Grimm \& Hochbruck.
\end{remark}


\subsection{A spectral Galerkin method for the discretization in space}\label{subsec-full}

For a full discretization of \eqref{eq-qlw-compact}, we combine the trigonometric integrators of the previous section with a spectral Galerkin method in space.

We denote by 
\[
\mathcal{V}^K = \Biggl\{ \, \sum_{j=-K}^{K} \hat{v}_j \e^{\iu jx} : \hat{v}_j\in\C \,\Biggr\}
\]
the space of trigonometric polynomials of degree $K$ and by 
\begin{equation}\label{eq-Pc}
\mathcal{P}^K(v) = \sum_{j=-K}^{K} \hat{v}_j \e^{\iu jx} \myfor v = \sum_{j=-\infty}^{\infty} \hat{v}_j \e^{\iu jx} \in L^2
\end{equation}
the $L^2$-orthogonal projection onto this ansatz space. In the semi-discretization in time \eqref{eq-method} or \eqref{eq-method-onestep}, we then replace the nonlinearity $\widehat{f}$ of \eqref{eq-fhat} by
\begin{equation}\label{nonf}
\widehat{f}^K(u) = \mathcal{P}^K \klabig{ \Psi_1 f^K(\Phi u) }
\end{equation}
with 
\begin{equation}\label{fK}
f^K(u) = a^K(u) \partial_x^2 u + g^K(u,\partial_x u), \qquad a^K = \mathcal{I}^K \circ a, \qquad g^K = \mathcal{I}^K \circ g,
\end{equation}
where $ \mathcal{I}^K$ denotes the trigonometric interpolation in the space $\mathcal{V}^K$ of trigonometric polynomials of degree $K$.

This gives the fully discrete method
\begin{equation}\label{eq-method-full}\begin{split}
u_{n+1}^K &= \cos(\tau\Om) u_n^K + \tau \sinc(\tau\Om) \dot{u}_n^K + \tfrac12 \tau^2 \sinc(\tau\Om) \ka \widehat{f}^K(u_n^K),\\
\dot{u}_{n+1}^K &= -\Om \sin(\tau\Om) u_n^K + \cos(\tau\Om) \dot{u}_n^K + \tfrac12 \tau \cos(\tau\Om) \ka \widehat{f}^K(u_n^K) + \tfrac12 \tau \ka \widehat{f}^K(u_{n+1}^K),
\end{split}\end{equation}
which computes approximations $u_n^K\in\mathcal{V}^K$ and $\dot{u}_n^K\in\mathcal{V}^K$ to $u(\cdot,t_n)$ and $\partial_t u(\cdot,t_n)$, respectively.
In addition, we replace the initial values $u_0$ and $\dot{u}_0$ of \eqref{eq-qlw-init} by some approximations $u_0^K\in\mathcal{V}^K$ and $\dot{u}_0^K\in\mathcal{V}^K$, computed by an $L^2$-orthogonal projection onto $\mathcal{V}^K$: 
\[
u_0^K = \Pc^K(u_0), \qquad \dot{u}_0^K = \Pc^K(\dot{u}_0).
\]

We emphasize that the nonlinearity $\widehat{f}^K$ as appearing in \eqref{eq-method-full} can be computed efficiently using fast Fourier techniques: The functions $a^K=\mathcal{I}^K \circ a$ and $g^K=\mathcal{I}^K \circ g$ can be computed as usual with the fast Fourier transform. The full nonlinearity $\widehat{f}^K$ of \eqref{nonf} can then also be computed with fast Fourier techniques, even though it is defined via projection instead of trigonometric interpolation. This is based on the observation that the argument of the projection $\mathcal{P}^K$ in \eqref{nonf} as appearing in \eqref{eq-method-full} is a trigonometric polynomial of degree $2K$, and hence
\[
\widehat{f}^K(v^K) = \mathcal{P}^K \klaBig{ \mathcal{I}^{2 K} \klabig{\Psi_1 f^K(\Phi v^K) }}
\]
with the trigonometric interpolation $\mathcal{I}^{2 K}$ in the larger space $\mathcal{V}^{2 K}$ of trigonometric polynomials of degree $2 K$.

\section{Statement of global error bounds}\label{sec-results}

In this section, we state our error bounds for the trigonometric integrator \eqref{eq-method} and its fully discrete version \eqref{eq-method-full} when applied to the quasilinear wave equation \eqref{eq-qlw}. 

We will universally require Assumptions \ref{assum_slw_psibds}--\ref{assum_lkap} on the filter functions of the trigonometric integrator \eqref{eq-method}. In addition, we will require that the exact solution $u(x,t)$ to \eqref{eq-qlw} satisfies the following assumption. 

\begin{assum}\label{assum-exact}
  Let $s\ge 0$. We assume that the exact solution $(u(\cdot,t),\partial_t u(\cdot,t))$ to \eqref{eq-qlw} is in $H^{5+s}\times H^{4+s}$ with
\begin{equation}\label{eq-regularity}
\normvbig{\klabig{u(\cdot,t) ,\partial_t u(\cdot,t)} }_{4+s} \le M \myfor 0\le t \le T
\end{equation}
 such that
\begin{equation}\label{eq-delta}
1+ \ka \, a(u(\cdot,t)) \geq \delta > 0 \qquad \text{for}\qquad 0\le t\le T
\end{equation}
and 
\begin{equation}\label{eq-A0}
\ka \, a(u(\cdot,t)) \le A_0 \qquad \text{for}\qquad 0\le t\le T
\end{equation}
for some constants $0<\de<1$, $M>0$ and $A_0\ge 0$.
\end{assum}

\begin{remark}
The restriction \eqref{eq-delta} in Assumption \ref{assum-exact} is a natural assumption coming from the analysis of the equation. It ensures the hyperbolic character of the  equation. 
By local well-posedness theory, the regularity assumption \eqref{eq-regularity} (which implies \eqref{eq-A0}) on the exact solution holds locally in time for initial values in $H^{5+s}\times H^{4+s}$, see \cite{Hughes1976,Kato1975,Taylor1991}. The time-scale of our numerical analysis is the time-scale where a solution to \eqref{eq-qlw} actually exists and stays bounded. 
\end{remark}

We are now ready to state the main result for the semi-discretization in time \eqref{eq-method} whose proof is given in Section \ref{sec-error} below.

\begin{theorem}[Error bound for the semi-discretization in time]\label{thm-large}
  Fix $M>0$, $T>0$, $c_0\ge 0$, $0<\de<1$ and $A_0\ge 0$. 
Then, there exists a positive constant $\tau_0$ such that, for all time step-sizes $\tau\le\tau_0$, the following global error bound holds for the time-discrete numerical solution $(u_n,\dot{u}_n)$ of \eqref{eq-method}. 

If the exact solution $(u(\cdot,t),\partial_t u(\cdot,t))$ satisfies Assumption \ref{assum-exact} for $s=0$ with constants $M$, $T$, $\de$ and $A_0$, and if the filter functions in \eqref{eq-method} satisfy Assumption \ref{assum_slw_psibds}--\ref{assum_lkap} with constants $c_0$, $\de$ and $A_0$, 
then we have in $H^2\times H^1$ the global error bound
\[
\normvbig{\klabig{u_n,\dot{u}_n} - \klabig{u(\cdot,t_n),\partial_t u(\cdot,t_n)} }_{1} \le C \tau^2 \myfor 0\le t_n=n\tau \le T.
\]
The constant $C$ is of the form $C=C' \e^{C' \abs{\ka} T}$ with $C'$ depending on $\max(1,\abs{\kappa})$ with the coefficient $\kappa$ in \eqref{eq-qlw}, the smooth functions $a$ and $g$ in \eqref{eq-qlw}, the constant $c_0$ of \eqref{eq-assumfilter-bounds}, the constants $\de$ and $A_0$ of \eqref{eq-assumfilter-largedata}, \eqref{eq-delta} and \eqref{eq-A0} and $M$ from \eqref{eq-regularity}, but it is independent of the time step-size $\tau$, the final time $T$ and the parameter $\ka$ in \eqref{eq-qlw}. 
\end{theorem} 

\begin{remark}
  For small nonlinearities with $|\kappa| \ll 1$, we need only Assumptions \ref{assum_slw_psibds} and \ref{assum_qlw_psi} on the filter functions of the trigonometric integrator \eqref{eq-method} to prove the global error bound, as explained in Remark \ref{rem3pt1}. In this case, the necessary energy estimates can be proved and bounded in a simpler fashion and the underlying ellipticity of the second order operator is more easily proved on long time scales.  We will give some indication of these simplifications in Section \ref{sec-error}.  
\end{remark} 

For the fully discrete trigonometric integrator \eqref{eq-method-full}, we will prove in Section \ref{sec-error-full} the following global error bound.

\begin{theorem}[Error bound for the full discretization]\label{thm-full}
  Fix $M>0$, $T>0$, $c_0\ge 0$, $0<\de<1$, $A_0\ge 0$ and $s\ge 0$. 
Then, there exists a positive constant $\tau_0$ such that, for all time step-sizes $\tau\le\tau_0$, the following global error bound holds for the fully discrete numerical solution $(u^K_n,\dot{u}^K_n)$ of \eqref{eq-method-full}.

If the exact solution $(u(\cdot,t),\partial_t u(\cdot,t))$ satisfies Assumption \ref{assum-exact} for the above $s$ with constants $M$, $T$, $\de$ and $A_0$, and if the filter functions in \eqref{eq-method} satisfy Assumption \ref{assum_slw_psibds}--\ref{assum_lkap} with constants $c_0$, $\de$ and $A_0$, 
then we have in $H^2\times H^1$ the global error bound
\[
\normvbig{\klabig{u_n^K,\dot{u}_n^K} - \klabig{u(\cdot,t_n),\partial_t u(\cdot,t_n)} }_{1} \le C \tau^2 + C K^{-2-s} \myfor 0\le t_n=n\tau \le T.
\]
The constant $C$ is of the same form as in Theorem \ref{thm-large} with $C'$ depending in addition on $s$.
\end{theorem}

The convergence rate $\tau^2$ in Theorem \ref{thm-full} for the discretization in time is optimal as will be shown in the following numerical examples. It is not clear whether the convergence rate $K^{-2-s}$ for the discretization in space is also optimal under the given regularity assumption. In fact, numerical experiments suggest that the error behaves like $K^{-3-s}$ almost uniformly in the time step-size.

In the following numerical examples, we consider the trigonometric integrator \eqref{eq-method} (or \eqref{eq-method-full}) with 
\begin{itemize}
\item no additional filter functions, i.e., $\phi=\psi_1=1$, which is known as impulse method or method of Deuflhard,  
\item filter functions \eqref{eq-trigoHL}, which is the method of Hairer \& Lubich and coincides with the new method \eqref{filt} for $c=0$,
\item filter functions \eqref{eq-trigoGH}, which is the method of Grimm \& Hochbruck and coincides with the new method \eqref{filt} for $c=1$,
\item filter functions \eqref{filt} with   $c=2$ and $c=3$,   which is the new method proposed in this paper for non-small nonlinearities.
\end{itemize}
The specific quasilinear wave equation that we consider is \eqref{eq-qlw} with $a(u)=u$ and $g(u,\partial_xu)=(\partial_x u)^2 + \kappa u^3$:
\begin{equation}\label{eq-qlw-ex}
\partial_{t}^2 u = \partial_{x}^2 u - u + \ka u \partial_{x}^2 u + \ka (\partial_x u)^2 + \ka^2 u^3.
\end{equation}
This is the model problem of \cite{Chong2013}. As initial values we consider rather artificially
\[
u(x,0) = \sum_{j\in\mathbb{Z}} \frac{1}{\sqrt{1+\abs{j}^{11+1/50}}} \e^{\iu jx}, \qquad \partial_t u(x,0) = \sum_{j\in\mathbb{Z}} \frac{1}{\sqrt{1+\abs{j}^{9+1/50}}} \e^{\iu jx}.
\] 
For this choice, the initial values are in $H^5\times H^4$, but not in $H^{5+\sigma}\times H^{4+\sigma}$ for $\sigma\ge 1/100$, so that the initial values just don't fail to satisfy the regularity assumption \eqref{eq-regularity} for $s=0$.

\begin{figure}
\centering \includegraphics{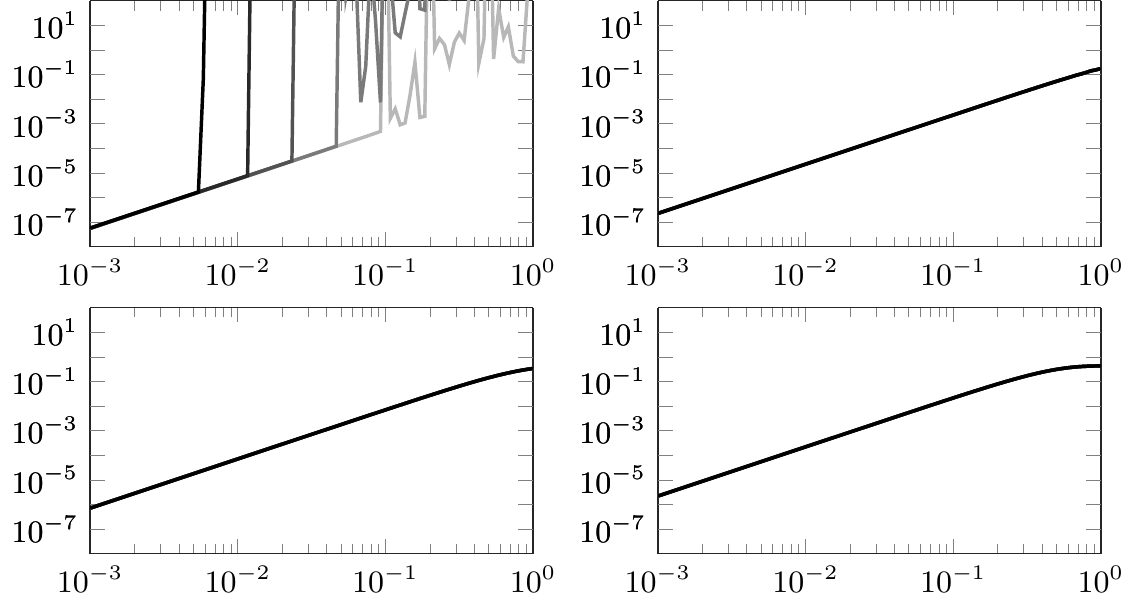}
\caption{Error in $H^2\times H^1$ at time $100$ vs.\ time step-size $\tau$ for a small nonlinearity ($\ka=1/100$). The methods are the impulse method (top left), the method \eqref{eq-trigoHL} of Hairer \& Lubich (top right), the method \eqref{eq-trigoGH} of Grimm \& Hochbruck (bottom left) and our new method \eqref{filt} with   $c=2$   (bottom right). Different lines correspond to different values of the discretization parameter $K=2^5,2^6,2^7,2^8,2^9$, with darker lines for larger values of $K$.}\label{fig-small}
\end{figure}

\begin{example}[Small nonlinearity]
We consider equation \eqref{eq-qlw-ex} with a small nonlinearity as in \cite{Chong2013}. We choose $\ka = 1/100$, and we consider correspondingly a long time interval of length $\ka^{-1}=100$. The error in $H^2\times H^1$ of various trigonometric integrators at time $t=100$ has been plotted in Figure \ref{fig-small}. In the plots, only the temporal error has been taken into account by comparing the numerical solution 
to a reference solution with the same spatial discretization parameter. 

For the method \eqref{eq-trigoHL} of Hairer \& Lubich, the method \eqref{eq-trigoGH} of Grimm \& Hochbruck and the new method \eqref{filt} with   $c=2$,   we observe second-order convergence in time uniformly in the spatial discretization parameter (the different lines corresponding to different spatial discretization parameters are all on top of each other).   Note that the filter functions of these methods satisfy Assumptions \ref{assum_slw_psibds} and \ref{assum_qlw_psi} of Theorems \ref{thm-large} and~\ref{thm-full} and that Assumption \ref{assum_lkap} is an empty condition for small $\kappa$ (see Remark \ref{rem3pt1}). The observed convergence of these methods can thus be explained with Theorems \ref{thm-large} and~\ref{thm-full}.   

For the impulse method, whose filter functions don't satisfy Assumption \ref{assum_qlw_psi}, we observe second-order convergence only for time step-sizes that are sufficiently small compared to the inverse of the spatial discretization parameter $K$.   This observation is explained, for small $\abs{\kappa}$, in a previous version of this paper \cite[Section 5.2]{Prev}.   It is a quasilinear phenomenon which is not present in the semilinear case \cite{Gauckler2015}.
\end{example}

\begin{figure}[t]
\centering \includegraphics{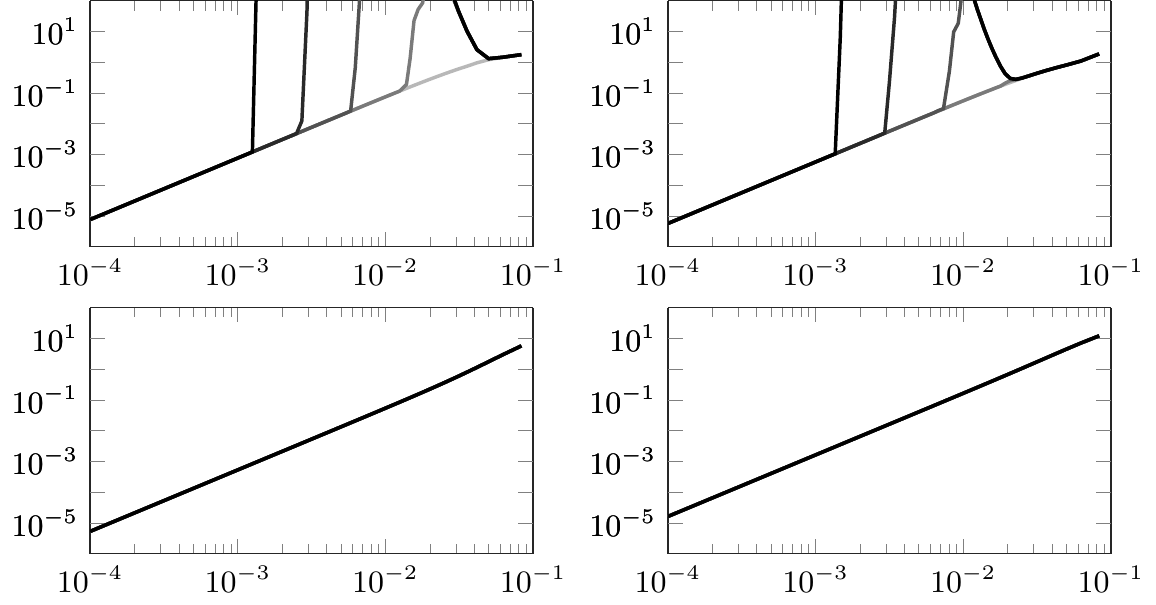}
\caption{Error in $H^2\times H^1$ at time $\frac14$ vs.\ time step-size $\tau$ for a non-small nonlinearity ($\ka=1$). The methods are the method \eqref{eq-trigoHL} of Hairer \& Lubich (top left), the method \eqref{eq-trigoGH} of Grimm \& Hochbruck (top right) and our new method \eqref{filt} with   $c=2$   (bottom left) and   $c=3$   (bottom right). Different lines correspond to different values of the discretization parameter $K=2^5,2^6,2^7,2^8,2^9$, with darker lines for larger values of $K$.}\label{fig-nonsmall}
\end{figure}

\begin{example}[Non-small nonlinearity]
We consider again equation \eqref{eq-qlw-ex}, but now with a non-small nonlinearity with $\ka = 1$ on a time interval of length $\frac14$. Numerical experiments suggest that the exact solution develops a singularity slightly beyond this time interval, and that $u=\ka\,a(u)$ (which appears in assumption~\eqref{eq-A0}) is bounded on this time interval by $A_0=13$. The error in $H^2\times H^1$ at time $t=\frac14$ of the methods has been plotted in Figure \ref{fig-nonsmall}. 

For the new method \eqref{filt} with   $c=2$ and $c=3$,   we observe second-order convergence in time. These methods satisfy Assumption \ref{assum_slw_psibds} and \ref{assum_qlw_psi}, but they also satisfy the additional Assumption \ref{assum_lkap} for non-small nonlinearities   with the relevant value $A_0=13$ (this follows from Remark \ref{rem3pt2} since $2>\tfrac12\sqrt{A_0}$). The observed convergence of the new methods can thus be explained with Theorems \ref{thm-large} and \ref{thm-full}.     In practice, one will choose filter functions $\phi$ as in \eqref{filt} with a value of $c$ as small as possible (because the error constant deteriorates as $c$ grows), and one will possibly adapt the value of $c$ in the course of the computation (depending on the size of the numerical solution).   

The methods \eqref{eq-trigoHL} of Hairer \& Lubich and \eqref{eq-trigoGH} of Grimm \& Hochbruck don't satisfy Assumption \ref{assum_lkap} with the necessary value $A_0=13$, although their filters are of the form \eqref{filt} of Remark \ref{rem3pt2}, but with a too small value $c=0$ and $c=1$, respectively.
For these methods, the observed convergence is not uniform in the spatial discretization parameter $K$. 
\end{example}

For additional numerical examples in connection with the questions studied in \cite{Groves2005,Chong2013,Chirilus-Bruckner2015,Duell2016}, we refer to a previous version of this paper \cite[Section 3.5]{Prev}.

\section{Proof of the error bound for the semi-discretization in time}
\label{sec-error}

In this section, we give the proof of Theorem \ref{thm-large} on the global error of the trigonometric integrator \eqref{eq-method} applied to the quasilinear wave equation \eqref{eq-qlw} without discretization in space. 
In the proof, we restrict to the case $g\equiv 0$ in \eqref{eq-qlw}, i.e., 
\begin{equation}\label{defF}
f(u) = a(u) \partial_{x}^2 u, \qquad \widehat{f}(u) = \Psi_1 \bigl( a(\Phi u) (\Phi \partial_{x}^2  u) \bigr) 
\end{equation}
in the notation \eqref{eq-f} and \eqref{eq-fhat}. Since the quasilinear term $a(u)\partial_{x}^2 u$ is the most critical part of the nonlinearity, the extension to nonzero $g$ is rather straightforward and we will comment throughout on the necessary modifications to take $g \neq 0$.  

Throughout the proof, we denote by $C$ a generic constant that may depend on  $a$,  an  upper bound $ \max(1,\abs{\kappa})$ on the absolute value of the coefficient $\ka$ in \eqref{eq-qlw} (but not on a lower bound), the order of the Sobolev space under consideration and on the constants $c_0$, $\de$ and $A_0$ of \eqref{eq-assumfilter-bounds} and \eqref{eq-assumfilter-largedata}. Additional dependencies of $C$ are denoted by lower indices, e.g., $C_M$ with $M$ from \eqref{eq-regularity}.

\subsection{Basic estimates}

The estimates \eqref{eq-algebra} and the smoothness of $a$  imply the following fundamental properties of the nonlinearity $f$ of \eqref{defF}: We have, for $s\ge 0$ and $u,v\in H^{s+2}$,
\begin{equation}\label{eq-algebra-f}
\norm{f(u)}_s \le \Lambda_{s}(\|u\|_{\sigma}) \|u\|_{\sigma} \norm{u}_{s+2}  \mywith \si = \max(s,1)
\end{equation}
and the Lipschitz property
\begin{equation}\label{eq-lipschitz}
\norm{f(u)-f(v)}_s \le \Lambda_{s}  \bigl( \norm{u}_{s+2} + \norm{v}_{s+2} \bigr)  \bigl( \norm{u}_{s+2} + \norm{v}_{s+2} \bigr) \norm{u-v}_{s+2},
\end{equation}
  where $\Lambda_{s}(\cdot)$ is a continuous non-decreasing function. 

Throughout the proof of Theorem \ref{thm-large}, we make use of the fact that the numerical flow $\ph_\tau$ given by \eqref{eq-method} maps $H^2\times H^1$ to itself and more generally $H^{s+1}\times H^s$ to itself for $s\ge 1$, as stated in the following lemma. This property of an explicit numerical method is in the quasilinear case by no means natural. It can be shown here using the smoothing properties of filter functions that satisfy \eqref{eq-assumfilter-psi1phi}.

\begin{lemma}[Bounds for a single time step]\label{lemma-singlestep}
Let $s\ge 1$, and let the filter functions satisfy Assumptions \ref{assum_slw_psibds} and \ref{assum_qlw_psi}. For a numerical solution $(u_{n},\dot{u}_{n}) \in H^{s+1}\times H^s$ with
\[
\normv{(u_n,\dot{u}_n)}_{s}\le M
\]
we have $(u_{n+1},\dot{u}_{n+1}) \in H^{s+1}\times H^s$ with
\[
\normv{(u_{n+1},\dot{u}_{n+1})}_{s} \le C_M.
\]
\end{lemma}
\begin{proof}
We consider the method in the form \eqref{eq-method-onestep}. In this formulation, we use that
\begin{equation}\label{eq-proofsinglestep-aux}
\tau\norm{\sinc(\tau\Om) u}_{s+1} = \norm{\Om^{-1}\sin(\tau\Om) u}_{s+1} = \norm{\sin(\tau\Om)u}_s \le \norm{u}_s.
\end{equation}
and that
\begin{equation}\label{eq-est-aux}
\tau\norm{\widehat{f}(u)}_{s} = \tau \norm{\Psi f(\Phi u)}_{s} \le \norm{\Phi f(\Phi u)}_{s-1} \le \Lambda_{s-1}( \|u\|_{s+1}) \|u\|_{s+1}^2.
\end{equation}
The first estimate of \eqref{eq-est-aux} follows from \eqref{eq-assumfilter-psi1phi} and \eqref{eq-proofsinglestep-aux} with $s$ instead of $s+1$ and the second one from \eqref{eq-assumfilter-bounds} and \eqref{eq-algebra-f}. 
These properties yield the claimed bound on the numerical solution, first for $u_{n+1}$ and with this also for $\dot{u}_{n+1}$.
\end{proof}

In the same way, but using the Lipschitz property \eqref{eq-lipschitz} instead of \eqref{eq-algebra-f}, we can derive the following estimate for the difference of two numerical solutions \eqref{eq-method}.

\begin{lemma}[Stability of a single time step]\label{lemma-stability-naive}
Let $s\ge 1$, and let the filter functions satisfy Assumptions \ref{assum_slw_psibds} and \ref{assum_qlw_psi}. For numerical solutions $(u_{n},\dot{u}_{n}) \in H^{s+1}\times H^s$ and $(v_{n},\dot{v}_{n}) \in H^{s+1}\times H^s$ with
\[
\normv{(u_n,\dot{u}_n)}_{s}\le M \myand \normv{(v_n,\dot{v}_n)}_{s}\le M
\]
we have 
\[
\normvbig{(u_{n+1},\dot{u}_{n+1}) - (v_{n+1},\dot{v}_{n+1})}_s \le C_M \normvbig{(u_{n},\dot{u}_{n}) - (v_{n},\dot{v}_{n})}_s . \tag*{\qed}
\]
\end{lemma}

\begin{remark}
  The basic estimates \eqref{eq-algebra-f} and \eqref{eq-lipschitz} extend directly to a nonzero $g$ in~\eqref{eq-qlw}, and hence also the statements of Lemmas \ref{lemma-singlestep} and \ref{lemma-stability-naive}.
\end{remark}

\subsection{Outline of the proof of Theorem \ref{thm-large}}\label{subsec-outline}

Lemmas \ref{lemma-singlestep} and in particular Lemma \ref{lemma-stability-naive} illustrate the difficulties that are encountered when trying to prove error bounds, say in $H^2\times H^1$. Denoting by 
\begin{equation}\label{eq-globalerror}
(e_{n+1},\dot{e}_{n+1}) = (u_{n+1} , \dot{u}_{n+1}) - (u(\cdot,t_{n+1}),\partial_t u(\cdot,t_{n+1}))
\end{equation}
the global error of \eqref{eq-method} after $n$ time steps, and denoting, with the numerical flow $\ph_\tau$ given by \eqref{eq-numflow}, by 
\begin{equation}\label{eq-localerror}
(d_{n+1},\dot{d}_{n+1}) = \ph_\tau\klabig{u(\cdot,t_{n}),\partial_t u(\cdot,t_{n})} - \klabig{u(\cdot,t_{n+1}),\partial_t u(\cdot,t_{n+1})}
\end{equation}
the local error when starting at $(u(\cdot,t_{n}),\partial_t u(\cdot,t_{n}))$, one routinely decomposes the global error in $H^2\times H^1$ as
\[
\normv{(e_{n+1},\dot{e}_{n+1})}_1 \le \normvbig{\ph_\tau\klabig{u_{n},\dot{u}_{n}} - \ph_\tau\klabig{u(\cdot,t_{n}),\partial_t u(\cdot,t_{n})}}_1 + \normv{(d_{n+1},\dot{d}_{n+1})}_1.
\]
By analyzing the error propagation of the method (stability of the method), one then aims for estimating the first term on the right-hand side by $\e^{C\tau} \normv{(e_{n},\dot{e}_{n})}_1$. The stability estimate of Lemma \ref{lemma-stability-naive}, however, only yields a factor $C_M$ instead of $\e^{C\tau}$, which makes this approach failing.

Our approach here is to replace the $H^2\times H^1$-norm by a different but related measure for the error that allows us to prove a suitable stability estimate. This measure isn't a norm, and it depends on time. Its definition is inspired by energy techniques as used to analyze the exact solution: We introduce in Section \ref{subsec-stability-small} an energy-type functional $\Ec\colon H^{2}\times H^1 \times H^2 \rightarrow \R$, and we will then use 
\[
\Ec_n\kla{e_n,\dot{e}_n } = \Ec \kla{e_n,\dot{e}_n,u_n}
\]
instead of $\normv{(e_{n},\dot{e}_{n})}_1$ as a measure for the global error $(e_n,\dot{e}_n)$.

The error accumulation in this quantity reads
\begin{equation}\label{eq-erroraccum}\begin{split}
&\Ec_{n+1}\kla{e_{n+1},\dot{e}_{n+1} } =
\Ec_{n+1}\klaBig{\ph_\tau\klabig{u_{n},\dot{u}_{n}} - \ph_\tau\klabig{u(\cdot,t_{n}),\partial_t u(\cdot,t_{n})}} \\
 &\qquad\qquad\qquad + \klaBig{ \Ec_{n+1}\kla{e_{n+1},\dot{e}_{n+1}} - \Ec_{n+1}\klabig{e_{n+1}-d_{n+1},\dot{e}_{n+1}-\dot{d}_{n+1}}} .
\end{split}\end{equation}
The difference in the second line of \eqref{eq-erroraccum} accounts for the local error of the method. It is estimated in Section \ref{subsec-locerror} below by adapting the proof for the semilinear case to the quasilinear case. The term in the first line is $\Ec_{n+1}$ evaluated at a difference of two numerical solutions, and hence describes the error propagation of the method. It turns out, in Section \ref{subsec-stability-small} below, that we can prove a suitable estimate for the error propagation in the quantity $\Ec$. The relation of $\Ec$ to the $H^2\times H^1$-norm is then described in Section \ref{subsec-sobolevnorm-small} below. Finally, the error accumulation is studied in Section \ref{subsec-accumulation-small} below.

\subsection{Stability of the numerical method}\label{subsec-stability-small}

A key step in the proof of Theorem \ref{thm-large} is to establish stability of the numerical method \eqref{eq-method} in a suitable sense. 

We introduce the energy-type quantity
\begin{equation}\label{eq-energy}
\Ec\kla{e,\dot{e},u} = \normv{(e,\dot{e})}_1^2 + \ka\, \Uc\klabig{\Phi e,\Phi u},
\end{equation}
where
\begin{equation}\label{Uc}
  \Uc \kla{ e, u}
= \sklabig{ \cos(\tau\Om) \partial_{x}^2 e,  a(u) \partial_{x}^2 e}_0 - \tfrac{1}{4}\tau^2 \kappa \normbig{ \Psi_1 \klabig{a(u) \partial_{x}^2 e} }_1^2.
\end{equation}
Up to the non-quadratic term $\Uc$, this energy $\Ec$ is essentially the $H^{2}\times H^1$-norm of $(e,\dot{e})$. Under the Assumptions \ref{assum_slw_psibds} and \ref{assum_qlw_psi}, the energy $\Ec$ is well-defined for $(e,\dot{e})\in H^{2}\times H^1$ and $u\in H^2$. This follows from   the Cauchy--Schwarz inequality and \eqref{eq-algebra} applied to the first term of $\Uc$ and from   assumptions \eqref{eq-assumfilter-bounds} and \eqref{eq-assumfilter-psi1phi} applied as in \eqref{eq-est-aux} to the second term of $\Uc$.

The motivation to define the energy as in \eqref{eq-energy} is the calculation in the proof of the following lemma, where we compute the change in the energy along differences of numerical solutions. 

\begin{lemma}[Change in the energy]\label{lemma-energy}
Let the filter functions satisfy Assumptions \ref{assum_slw_psibds} and \ref{assum_qlw_psi}. For numerical solutions $(u_n,\dot{u}_n)\in H^{2}\times H^1$ and $(v_n,\dot{v}_n)\in H^{2}\times H^1$ we then have
\begin{multline*}
\Ec\klabig{u_{n+1}-v_{n+1},\dot{u}_{n+1}-\dot{v}_{n+1},u_{n+1}}\\ = \Ec\klabig{u_n-v_n,\dot{u}_n-\dot{v}_n,u_n} + \ka\,   \Rc\klabig{\Phi u_{n+1},\Phi u_n,\Phi v_{n+1},\Phi v_n}
\end{multline*}
with the remainder   
\begin{equation}\label{eq-rem-all}
\Rc\klabig{u,u',v,v'} = \widetilde \Rc\klabig{u,u',v,v'} + \Rc^*\klabig{u,v} - \Rc^*\klabig{u',v'},
\end{equation}
where
\begin{equation}\label{rem}
\widetilde \Rc\klabig{u,u',v,v'} = \sklabig{u-v , f(u') - f(v') }_{1} - \sklabig{u'-v', f(u) - f(v) }_{1}
\end{equation}
and 
\begin{align*}
\Rc^*(u,v) &= 
 \sklabig{ \cos(\tau\Om) (u-v), a(u) \partial_x^2(u-v) }_0 + \sklabig{ \cos(\tau\Om) (u-v), \klabig{a(u)-a(v)} \partial_x^2v }_1 \\ 
&\qquad+ \tfrac{1}{2}\tau^2 \kappa \,\sklabig{ \Psi_1 \klabig{ a(u) \partial_{x}^2 (u-v)},  \Psi_1 \klabig{ (a(u) - a(v)) \partial_{x}^2 v} }_1\\
&\qquad +\tfrac{1}{4} \tau^2 \kappa \,\normbig{ \Psi_1 \klabig{(a(u) -a(v))\partial_{x}^2 v} }_1^2.
\end{align*}
\end{lemma}
\begin{proof}
By the structure of the ``matrix'' in the second step of the method \eqref{eq-method}, we have
\[
\norm{\Om (u_{n+1}-v_{n+1})}_1^2 + \norm{\dot{u}_{n+1}^- - \dot{v}_{n+1}^-}_1^2 = \norm{\Om (u_n-v_n)}_1^2 + \norm{\dot{u}_{n}^+-\dot{v}_{n}^+}_1^2. 
\]
Hence, taking also the first and third step of the method into account, we get
\begin{multline*}
\norm{\Om (u_{n+1}-v_{n+1})}_1^2 + \normbig{\klabig{\dot{u}_{n+1}-\dot{v}_{n+1}} - \tfrac12 \tau \ka \klabig{\widehat{f}(u_{n+1})-\widehat{f}(v_{n+1})}}_1^2\\
 = \norm{\Om (u_n-v_n)}_1^2 + \normbig{\klabig{\dot{u}_{n}-\dot{v}_{n}} + \tfrac12 \tau \ka \klabig{\widehat{f}(u_n)-\widehat{f}(v_n)} }_1^2.
\end{multline*}
We then expand the second norm on the left and on the right. In the resulting mixed terms, we use the property
\begin{equation}\label{eq:sincIP}
\skla{v,\Psi_1w}_1=\skla{\Psi_1v,w}_1=\skla{\sinc(\tau\Om)\Phi v,w}_1,
\end{equation}
which follows from Parseval's theorem and assumption \eqref{eq-assumfilter-psi1phi}, and we replace the resulting differences $\tau\sinc(\tau\Om)\kla{\dot{u}_{n+1}-\dot{v}_{n+1}}$ and $\tau\sinc(\tau\Om)\kla{\dot{u}_{n}-\dot{v}_{n}}$ with the help of the relations
\begin{align*}
\tau\sinc(\tau\Om) \dot{u}_{n+1} &= \cos(\tau\Om) u_{n+1} + \tfrac12 \tau^2 \sinc(\tau\Om) \ka \widehat{f}(u_{n+1}) - u_{n},\\
\tau\sinc(\tau\Om) \dot{u}_n &= -\cos(\tau\Om) u_n - \tfrac12 \tau^2 \sinc(\tau\Om) \ka \widehat{f}(u_n) + u_{n+1}
\end{align*}
and the same relations for $v$. The second of these relations is taken from \eqref{eq-method-onestep}, and the first one can be derived from the first one by the symmetry of the method (or from the numerical method in the form \eqref{eq-method} by expressing $u_n$ in terms of $u_{n+1}$ and $\dot{u}_{n+1}$). Using again \eqref{eq:sincIP}, the definition of the remainder $\Rc$ and $\normv{(e,\dot{e})}_1^2=\norm{\Om e}_1^2 + \norm{\dot{e}}_1^2$,   this yields
\begin{multline*}
\normvbig{\klabig{u_{n+1}-v_{n+1},\dot{u}_{n+1}-\dot{v}_{n+1}}}_1^2 + \kappa \,\widetilde{\Uc}\klabig{\Phi (u_{n+1}-v_{n+1}),\Phi u_{n+1}}\\ 
 = \normvbig{\klabig{u_{n}-v_{n},\dot{u}_{n}-\dot{v}_{n}}}_1^2 + \kappa \,\widetilde{\Uc}\klabig{\Phi (u_{n}-v_{n}),\Phi u_{n}}\\ + \ka\, \widetilde{\Rc}\klabig{\Phi u_{n+1},\Phi u_n,\Phi v_{n+1},\Phi v_n}
\end{multline*}
with $\widetilde{\Rc}$ from \eqref{rem} and
\[
\widetilde{\Uc} \kla{e,u} = - \sklabig{\cos(\tau\Om) e , f(u) - f(u-e) }_{1} - \tfrac14 \tau^2 \ka\, \normbig{\Psi_1 \klabig{f(u) - f(u-e)} }_{1}^2 . 
\]
The statement of the lemma follows by setting
\[
\Rc^*(u,v) = \Uc\kla{u-v, u} - \widetilde{\Uc}\kla{u-v,u}.
\]
To get the final form of $\Rc^*$, we use 
\[
f(u)-f(v) = a(u)\partial_x^2 (u-v) + \kla{a(u)-a(v)} \partial_x^2 v
\]
and $\skla{\cdot,\cdot}_1=\skla{\cdot,\cdot}_0-\skla{\partial_x^2\cdot,\cdot}_0$.
\end{proof}

We now estimate the remainder $\Rc$ of Lemma \ref{lemma-energy}, which describes the change in the energy along numerical solutions. The crucial observation is that we gain a factor $\tau$ without requiring more regularity than $H^2\times H^1$ of the difference of the corresponding numerical solutions. 

\begin{lemma}[Bound of the change $\Rc$ in the energy]\label{lemma-remainder-energyest}
Let the filter functions satisfy Assumptions \ref{assum_slw_psibds} and \ref{assum_qlw_psi}. For numerical solutions $(u_{n},\dot{u}_{n}) \in H^{2}\times H^1$ and $(v_{n},\dot{v}_{n}) \in H^{3}\times H^2$ with
\[
\normv{(u_n,\dot{u}_n)}_1\le M \myand \normv{(v_n,\dot{v}_n)}_2\le M
\]
we have for the remainder $\Rc$ of Lemma \ref{lemma-energy} the bound
\[
\absbig{\Rc\klabig{\Phi u_{n+1},\Phi u_n,\Phi v_{n+1},\Phi v_n}} \le C_M \tau \normvbig{(u_n,\dot{u}_n)-(v_n,\dot{v}_n)}_{1}^2 .
\]
\end{lemma}
\begin{proof}
The main task is to get the factor $\tau$ in the estimate. This is done with the observation that we have 
\begin{equation}\label{eq-remainderest-aux-u}
\norm{u_{n+1}-u_n}_1 \le C_M \tau,
\end{equation}
which follows from \eqref{eq-method-onestep} using the property \eqref{eq-est-aux} with $s=1$ and $\norm{(\cos(\tau\Om)-1)u_n}_1   = 2 \norm{\sin(\frac12\tau\Om)^2u_n}_1 \le \norm{\tau \Om u_n}_1 = \tau \norm{u_n}_{2}$. Similarly, we have
\begin{equation}\label{eq-remainderest-aux-v}
\norm{v_{n+1}-v_n}_2 \le C_M \tau
\end{equation}
by the higher regularity of $(v_n,\dot{v}_n)$  and also
\begin{equation}\label{eq-remainderest-aux-d}
\norm{u_{n+1}-u_n - (v_{n+1}-v_n)}_1 \le C_M \tau \normvbig{(u_n,\dot{u}_n)-(v_n,\dot{v}_n)}_{1}.
\end{equation}
 Under the given regularity assumptions, the differences $u_{n+1}-u_n$ in $H^1$ and $v_{n+1}-v_n$ in $H^2$ thus allow us to gain a factor $\tau$. 

Our goal is therefore to recover in the remainder $\Rc$ defined in \eqref{eq-rem-all} such differences.  In the following we set
\begin{equation}\label{uvPhi}
\widehat{u}_n = \Phi u_n, \quad \widehat{v}_{n} = \Phi v_n.
\end{equation}
In this notation  and using that $\skla{\cdot,\cdot}_{1} =\skla{\cdot,\cdot}_{0} + \skla{\partial_x \cdot,\partial_x \cdot}_{0}$ we can express the remainder $\Rc$ in the following way: We have
\[
\Rc\klabig{\widehat{u}_{n+1},\widehat{u}_{n},\widehat{v}_{n+1},\widehat{v}_{n}} =  \Rc_0 + \Rc_1 +   \klabig{ \Rc^*\klabig{\widehat{u}_{n+1},\widehat{v}_{n+1}} - \Rc^*\klabig{\widehat{u}_{n},\widehat{v}_{n}}},
\]
where
\begin{multline*}
\Rc_0 = \sklabig{\widehat{u}_{n+1}-\widehat{v}_{n+1} , a(\widehat{u}_n)\partial_x^2 \widehat{u}_n-a (\widehat{v}_{n}) \partial_x^2 \widehat{v}_{n} }_{0}\\ - \sklabig{\widehat{u}_n-\widehat{v}_{n} , a(\widehat{u}_{n+1})\partial_x^2 \widehat{u}_{n+1}-a (\widehat{v}_{n+1}) \partial_x^2 \widehat{v}_{n+1} }_{0}
\end{multline*}
and
\begin{multline*}
\Rc_1 = \sklabig{\partial_x(\widehat{u}_{n+1}-\widehat{v}_{n+1}) , \partial_x \left(a(\widehat{u}_n)\partial_x^2 \widehat{u}_n-a (\widehat{v}_{n}) \partial_x^2 \widehat{v}_{n}\right)}_{0}\\ - \sklabig{\partial_x(\widehat{u}_n-\widehat{v}_{n}) ,\partial_x\left( a(\widehat{u}_{n+1})\partial_x^2 \widehat{u}_{n+1}-a (\widehat{v}_{n+1}) \partial_x^2 \widehat{v}_{n+1} \right)}_{0}.
\end{multline*}

With the aid of integration by parts and adding zeroes we obtain that
\begin{multline*}
\Rc_1 = - \sklabig{\partial_x^2\big(\widehat{u}_{n+1}-\widehat{v}_{n+1}\big) , a(\widehat{u}_n)\partial_x^2 \big(\widehat{u}_n-\widehat{v}_{n})+ \big(a(\widehat{u}_n)-a (\widehat{v}_{n}) \big)\partial_x^2 \widehat{v}_{n}}_{0}\\ + \sklabig{\partial_x^2 \big(\widehat{u}_n-\widehat{v}_{n}\big) , a(\widehat{u}_{n+1})\partial_x^2 \big(\widehat{u}_{n+1}-\widehat{v}_{n+1}\big)+ \big(a(\widehat{u}_{n+1})-a (\widehat{v}_{n+1})\big) \partial_x^2 \widehat{v}_{n+1}}_{0}.
\end{multline*}
Note that by symmetry we have for the first term that
\begin{multline*}
\sklabig{\partial_x^2\big(\widehat{u}_{n+1}-\widehat{v}_{n+1}\big) , a(\widehat{u}_n)\partial_x^2 \big(\widehat{u}_n-\widehat{v}_{n})}_0 =\sklabig{\partial_x^2\big(\widehat{u}_{n}-\widehat{v}_{n}\big) , a(\widehat{u}_n)\partial_x^2 \big(\widehat{u}_{n+1}-\widehat{v}_{n+1})}_0
\end{multline*}
which we can combine with the first term in the second row, i.e., 
\begin{multline*}
\Rc_1 = \sklabig{\partial_x^2\big(\widehat{u}_{n+1}-\widehat{v}_{n+1}\big) , \big(a(\widehat{u}_{n+1})- a(\widehat{u}_n)\big)\partial_x^2 \big(\widehat{u}_n-\widehat{v}_{n})}_0\\
-\sklabig{\partial_x^2\big(\widehat{u}_{n+1}-\widehat{v}_{n+1}\big) , \big(a(\widehat{u}_n)-a (\widehat{v}_{n}) \big)\partial_x^2 \widehat{v}_{n}}_{0}\\ + \sklabig{\partial_x^2 \big(\widehat{u}_n-\widehat{v}_{n}\big) ,  \big(a(\widehat{u}_{n+1})-a (\widehat{v}_{n+1})\big) \partial_x^2 \widehat{v}_{n+1}}_{0}.
\end{multline*}
Adding and subtracting the term $\sklabig{\partial_x^2\big(\widehat{u}_n-\widehat{v}_{n} \big) , \big(a(\widehat{u}_n)-a (\widehat{v}_{n}) \big)\partial_x^2 \widehat{v}_{n}}_{0}$ (and combining it with the term in the second row) furthermore yields that
\begin{multline*}
\Rc_1 =  \sklabig{\partial_x^2\big(\widehat{u}_{n+1}-\widehat{v}_{n+1}\big) , \big(a(\widehat{u}_{n+1})- a(\widehat{u}_n)\big)\partial_x^2 \big(\widehat{u}_n-\widehat{v}_{n})}_0\\
+\sklabig{\partial_x^2\big(\widehat{u}_n-\widehat{v}_{n} - (\widehat{u}_{n+1}-\widehat{v}_{n+1})\big) , \big(a(\widehat{u}_n)-a (\widehat{v}_{n}) \big)\partial_x^2 \widehat{v}_{n}}_{0}\\ 
-\sklabig{\partial_x^2\big(\widehat{u}_n-\widehat{v}_{n} \big) , \big(a(\widehat{u}_n)-a (\widehat{v}_{n}) \big)\partial_x^2 \widehat{v}_{n}}_{0}\\ 
+ \sklabig{\partial_x^2 \big(\widehat{u}_n-\widehat{v}_{n}\big) ,\big(a(\widehat{u}_{n+1})-a (\widehat{v}_{n+1})\big) \partial_x^2 \widehat{v}_{n+1}}_{0}.
\end{multline*}
Finally, adding and subtracting the term $\sklabig{\partial_x^2 \big(\widehat{u}_n-\widehat{v}_{n}\big) ,\big(a(\widehat{u}_{n+1})-a (\widehat{v}_{n+1})\big) \partial_x^2 \widehat{v}_{n}}_{0} $ (and combining it with the terms in the last two rows) and using integration by parts on the term in the second row, we obtain that
\begin{multline*}
\Rc_1 =  \sklabig{\partial_x^2\big(\widehat{u}_{n+1}-\widehat{v}_{n+1}\big) , \big(a(\widehat{u}_{n+1})- a(\widehat{u}_n)\big)\partial_x^2 \big(\widehat{u}_n-\widehat{v}_{n})}_0\\
-\sklabig{\partial_x\big(\widehat{u}_n-\widehat{v}_{n} - (\widehat{u}_{n+1}-\widehat{v}_{n+1})\big) ,\partial_x\left( \big(a(\widehat{u}_n)-a (\widehat{v}_{n}) \big)\partial_x^2 \widehat{v}_{n}\right)}_{0}\\ 
+\sklabig{\partial_x^2\big(\widehat{u}_n-\widehat{v}_{n} \big) , \big( a(\widehat{u}_{n+1}) - a(\widehat{v}_{n+1}) - (a(\widehat{u}_n)-a (\widehat{v}_{n})) \big)\partial_x^2 \widehat{v}_{n}}_{0}\\ 
+ \sklabig{\partial_x^2 \big(\widehat{u}_n-\widehat{v}_{n}\big) ,\big(a(\widehat{u}_{n+1})-a (\widehat{v}_{n+1})\big) \partial_x^2 \big(\widehat{v}_{n+1}-\widehat{v}_{n}\big)}_{0}.
\end{multline*}
With the aid of the bilinear estimates \eqref{eq-algebra} we may thus bound the remainder $\Rc_1$ as follows: We have
\begin{equation}
\begin{aligned}\label{bound1-R1}
\vert \Rc_1 \vert &\leq C \Vert \widehat{u}_{n+1} - \widehat{v}_{n+1}\Vert_2 \Vert a(\widehat{u}_{n+1}) - a(\widehat{u}_n)\Vert_1 \Vert \widehat{u}_n-\widehat{v}_{n}\Vert_2\\
&\qquad +C\Vert \widehat{u}_{n+1}-\widehat{v}_{n+1} - (\widehat{u}_n-\widehat{v}_{n})\Vert_1 \Vert a(\widehat{u}_n)-a(\widehat{v}_{n})\Vert_1 \Vert \widehat{v}_{n}\Vert_3 \\
&\qquad +C \Vert \widehat{u}_n - \widehat{v}_{n}\Vert_2 \Vert a(\widehat{u}_{n+1})-a(\widehat{v}_{n+1}) - (a(\widehat{u}_n) - a(\widehat{v}_{n}))\Vert_1 \Vert \widehat{v}_{n}\Vert_2\\
&\qquad + C\Vert \widehat{u}_n-\widehat{v}_{n}\Vert_2 \Vert a(\widehat{u}_{n+1})-a(\widehat{v}_{n+1}) \Vert_1 \Vert \widehat{v}_{n+1}-\widehat{v}_{n}\Vert_2.
\end{aligned}
\end{equation}
To estimate the quadruple term in $a$ in the third line of \eqref{bound1-R1},   we consider the smooth function $H(u,e) = a(u+e)-a(u)$ and note that by \eqref{eq-algebra} and since $H(u,0)=0$
\begin{align*}
  \norm{H(u,e)-H(v,f)}_1 &\le \norm{H(u,e)-H(u,f)}_1 + \norm{H(u,f)-H(v,f)}_1\\
  &\le \Lambda\klabig{\norm{u}_1+\norm{v}_1+\norm{e}_1+\norm{f}_1} \klabig{ \norm{e-f}_1 + \norm{u-v}_1 \norm{f}_1}
\end{align*}
with a non-decreasing function $\Lambda(\cdot)$. With $u=\widehat{u}_n$, $e=\widehat{u}_{n+1}-\widehat{u}_n$, $v=\widehat{v}_n$ and $f=\widehat{v}_{n+1}-\widehat{v}_n$, this yields for the quadruple term in $a$ in the third line of \eqref{bound1-R1} 
\begin{multline}\label{quada}
  \normbig{ \klabig{a(\widehat u_{n+1})-a(\widehat u_{n})} - \klabig{a(\widehat v_{n+1})-a(\widehat v_{n})} }_1 \le \Lambda\klabig{2\norm{\widehat u_n}_1+2\norm{\widehat v_n}_1+\norm{\widehat u_{n+1}}_1+\norm{\widehat v_{n+1}}_1} \\
   \cdot \klaBig{ \normbig{\widehat u_{n+1}- \widehat u_n - (\widehat v_{n+1}-\widehat v_n)}_1 + \norm{\widehat u_{n}-\widehat v_{n}}_1 \norm{\widehat v_{n+1}-\widehat v_{n}}_1 }.
\end{multline}

Thanks to the bound \eqref{propdephi1} on the filter functions we obtain with the notation \eqref{uvPhi} that
\begin{equation}\label{uphiu}
\Vert \widehat{u}_n - \widehat{v}_{n} \Vert_s   =  \Vert \Phi (u_n-v_n)\Vert_s \leq \Vert u_n-v_n\Vert_s.
\end{equation}
Plugging the bounds  \eqref{quada} and \eqref{uphiu} together with the bounds on the difference  $\Vert u_{n+1}-u_n\Vert_1$ given in   \eqref{eq-remainderest-aux-u}, the difference $\Vert v_{n+1}-v_n\Vert_{  2 }$ given in   \eqref{eq-remainderest-aux-v} and the difference $\Vert u_{n+1}-u_n-(v_{n+1}-v_n)\Vert_1$ given in \eqref{eq-remainderest-aux-d} as well as the bound on $\Vert  u_{n+1}-v_{n+1}\Vert_{2}$ given in  Lemmas \ref{lemma-singlestep} and \ref{lemma-stability-naive} into \eqref{bound1-R1} yields that
\[
\vert \Rc_1\vert   \le C_M \tau \normvbig{(u_n,\dot{u}_n)-(v_n,\dot{v}_n)}_{1}^2,
\]
where we have used the given regularity assumptions (in particular that $v_n \in H^3$) and that $a$ is a sufficiently smooth function.

Thus, the proof is completed upon computing the comparable bound on the more regular terms $\Rc_0$   and $\Rc^*\kla{\widehat{u}_{n+1},\widehat{v}_{n+1}} - \Rc^*\kla{\widehat{u}_{n},\widehat{v}_{n}}$ by a similar analysis. For example, the difference $\Rc^*\kla{\widehat{u}_{n+1},\widehat{v}_{n+1}} - \Rc^*\kla{\widehat{u}_{n},\widehat{v}_{n}}$ contains the difference
\begin{multline*}
  \Rc^*_1 = \sklabig{ \cos(\tau\Om) (\widehat{u}_{n+1}-\widehat{v}_{n+1}), a(\widehat{u}_{n+1}) \partial_x^2(\widehat{u}_{n+1}-\widehat{v}_{n+1}) }_0 \\
  - \sklabig{ \cos(\tau\Om) (\widehat{u}_{n}-\widehat{v}_{n}), a(\widehat{u}_{n}) \partial_x^2(\widehat{u}_{n}-\widehat{v}_{n}) }_0,
\end{multline*}
which can be split as
\begin{multline*}
  \Rc^*_1
  = \sklabig{ \cos(\tau\Om) \klabig{\widehat{u}_{n+1}-\widehat{u}_n-(\widehat{v}_{n+1}-\widehat{v}_n)}, a(\widehat{u}_{n+1}) \partial_x^2(\widehat{u}_{n+1}-\widehat{v}_{n+1}) }_0\\
  + \sklabig{ \cos(\tau\Om) (\widehat{u}_{n}-\widehat{v}_{n}), \klabig{a(\widehat{u}_{n+1})-a(\widehat{u}_{n})} \partial_x^2(\widehat{u}_{n+1}-\widehat{v}_{n+1}) }_0\\
  + \sklabig{ \cos(\tau\Om) (\widehat{u}_{n}-\widehat{v}_{n}),  a(\widehat{u}_{n})\partial_x^2 \klabig{\widehat{u}_{n+1}-\widehat{u}_n-(\widehat{v}_{n+1}-\widehat{v}_n)} }_0.
\end{multline*}
After partial integration in the last term, this can be estimated as above by
\begin{multline*}
  \abs{\Rc^*_1}
  \le \norm{\widehat{u}_{n+1}-\widehat{u}_n-(\widehat{v}_{n+1}-\widehat{v}_n)}_0 \norm{a(\widehat{u}_{n+1})}_1 \norm{\widehat{u}_{n+1}-\widehat{v}_{n+1}}_2\\
  + \norm{\widehat{u}_{n}-\widehat{v}_{n}}_0 \norm{a(\widehat{u}_{n+1})-a(\widehat{u}_{n})}_1 \norm{\widehat{u}_{n+1}-\widehat{v}_{n+1} }_2\\
  + \norm{\widehat{u}_{n}-\widehat{v}_{n}}_1 \norm{a(\widehat{u}_{n})}_1 \norm{\widehat{u}_{n+1}-\widehat{u}_n-(\widehat{v}_{n+1}-\widehat{v}_n)}_1,
\end{multline*}
and hence
\[
\vert \Rc_1^*\vert   \le C_M \tau \normvbig{(u_n,\dot{u}_n)-(v_n,\dot{v}_n)}_{1}^2.
\]
Another exemplary term in the difference $\Rc^*\kla{\widehat{u}_{n+1},\widehat{v}_{n+1}} - \Rc^*\kla{\widehat{u}_{n},\widehat{v}_{n}}$ is 
\[
\Rc^*_2 = \tfrac{1}{4} \tau^2 \kappa \,\normbig{ \Psi_1 \klabig{(a(\widehat{u}_{n+1}) -a(\widehat{v}_{n+1}))\partial_{x}^2 \widehat{v}_{n+1}} }_1^2,
\]
for which we get with \eqref{eq-algebra}, \eqref{eq-assumfilter-bounds} and Lemmas \ref{lemma-singlestep} and \ref{lemma-stability-naive}
\begin{align*}
  \abs{\Rc^*_2} &\le C \tau^2 \normbig{ a(\widehat{u}_{n+1}) -a(\widehat{v}_{n+1})}_1^2 \norm{\widehat{v}_{n+1} }_3^2 \le C_M \tau \normvbig{(u_n,\dot{u}_n)-(v_n,\dot{v}_n)}_{1}^2.\qedhere
\end{align*}
\end{proof}

In the situation outlined in Section \ref{subsec-outline}, we get from Lemmas \ref{lemma-energy} and \ref{lemma-remainder-energyest} the following estimate.

\begin{proposition}[Stability]\label{prop-stability}
Let the filter functions satisfy Assumptions \ref{assum_slw_psibds} and \ref{assum_qlw_psi}. If $(u,\partial_t u)$ is a solution to \eqref{eq-qlw-compact} in $H^{3}\times H^2$ with 
\[
\normvbig{\klabig{u(\cdot,t_n),\partial_t u(\cdot,t_n)}}_2\le M,
\]
and if $(u_{n},\dot{u}_{n})\in H^2\times H^1$ is a corresponding numerical solution with
\[
\normv{(u_{n},\dot{u}_n)}_1 \le 2 M,
\]
then we have 
\begin{multline*}
\absBig{\Ec\klaBig{(u_{n+1},\dot{u}_{n+1})-\ph_\tau\klabig{u(\cdot,t_{n}),\partial_t u(\cdot,t_{n})},u_{n+1}}}\\
\le \absbig{\Ec\kla{e_n,\dot{e}_n,u_n}} + C_M \tau \abs{\ka} \, \normv{(e_n,\dot{e}_n)}_1^2
\end{multline*}
with the global error $(e_{n},\dot{e}_{n})$ of \eqref{eq-globalerror}. 
\end{proposition}
\begin{proof}
Take $(v_n,\dot{v}_n)=\kla{u(\cdot,t_n),\partial_t u(\cdot,t_n)}$ and $(v_{n+1},\dot{v}_{n+1})=\ph_\tau \kla{u(\cdot,t_n),\partial_t u(\cdot,t_n)}$ in Lemmas \ref{lemma-energy} and \ref{lemma-remainder-energyest}.
\end{proof}

\begin{remark}
  For a nonzero $g$ in~\eqref{eq-qlw}, the statement of Lemma \ref{lemma-energy} remains valid, but with a remainder $\mathcal{R}^*$ that contains additional terms with $g(u)$ instead of $a(u)\partial_x^2u$. As $g(u)$ is more regular than $a(u)\partial_x^2u$, the remainder estimate of Lemma \ref{lemma-remainder-energyest} extends to these new terms, and hence Proposition \ref{prop-stability} on the stability of the method also holds for nonzero $g$.
\end{remark}

\subsection{Controlling Sobolev norms with the energy}\label{subsec-sobolevnorm-small}

Our aim is to   show that the energy \eqref{eq-energy} can be controlled by Sobolev norms and vice-versa. This is done by estimating the additional contribution $\ka \, \Uc$ from above and below. For small $\abs{\kappa}$, this is elementary, and the main result of this section, Proposition \ref{lowerBoundU} below, can be derived directly from the properties \eqref{eq-assumfilter-bounds} and \eqref{eq-assumfilter-psi1phi} of the filters and \eqref{eq-algebra} of Sobolev spaces using \eqref{eq-proofsinglestep-aux}. For non-small $\kappa$, we have to work harder, and this is the main content of this section.

In the following, we set
\begin{align*}
 & \mathcal{L}(u)
= \ka \Phi a(u)\cos(\tau\Om) \Phi - \tfrac{1}{4} \kappa^2 \Phi a(u) \sin^2(\tau\Om) \Phi^2 a(u) \Phi. 
\end{align*}
Using integration by parts and $\skla{\sin(\tau\Omega)v,\sin(\tau\Omega)v}_0 = \tau^2\norm{\sinc(\tau\Omega)v}_1^2$, we obtain under the Assumption \ref{assum_qlw_psi} that 
\begin{equation}\label{repU}
\ka \, \Uc \kla{ \Phi e, \Phi u} = \sklabig{ \mathcal{L}(\Phi u) \partial_{x}^2 e, \partial_{x}^2 e }_0
\end{equation}
for the term $\Uc$ defined in \eqref{Uc}.
We have to prove an upper and a lower bound for this term, the latter being the crucial and difficult part. 
The key tool to prove the essential lower bound is given by the following lemma.

\begin{lemma}
\label{lem3}
Fix $M>0$, and  let the filter functions satisfy Assumptions \ref{assum_slw_psibds} and \ref{assum_lkap} with constants $0<\delta<1$ and $A_0\ge 0$.
Then,  there exists  $ \tau_{0}>0$ such that for every $\tau \le \tau_{0}$, for every  $v \in L^2$ and for every $u\in H^2$ with
\[
\norm{u}_2\le M,\qquad 1+  \kappa\, a(u(x)) \geq  \tfrac12 \delta > 0, \qquad   \kappa\, a( u(x) ) \le A_0 + \tfrac12 \delta, 
\]
we have that
\begin{equation*}
\Vert v \Vert_{0}^2  + \bigl\langle \mathcal{L}(\Phi u) v,v \bigr\rangle_{0} \geq \tfrac18 \delta \Vert v \Vert_{0}^2.
\end{equation*}
\end{lemma}

In order to prove the above Lemma we need the following estimate.

\begin{lemma}\label{lem4}
  Let the filter functions satisfy Assumptions \ref{assum_slw_psibds} and \ref{assum_lkap} with constants $0<\delta<1$ and $A_0\ge 0$.
We then have, for all $A\le A_0+\tfrac12\de$ with $1+ A\ge\tfrac12\delta>0$ and all $\xi\ge 0$,
\begin{equation}\label{eq-cossin-ineq}
A\cos(\xi) \phi(\xi)^2 - \tfrac14 A^2 \sin(\xi)^2 \phi(\xi)^4 \ge -1 + \tfrac12\delta.
\end{equation}
\end{lemma}
\begin{proof}
  We use $\cos(\xi) = \cos(\tfrac12\xi)^2-\sin(\tfrac12\xi)^2$ and $\sin(\xi)=2\cos(\tfrac12\xi)\sin(\tfrac12\xi)$ to rewrite \eqref{eq-cossin-ineq} as
\begin{equation}\label{eq-cossin-ineq-temp2}
\klabig{ 1 + A \cos(\tfrac12\xi)^2 \phi(\xi)^2} \klabig{ 1 - A \sin(\tfrac12\xi)^2 \phi(\xi)^2} \ge \tfrac12\de.
\end{equation}
As a function of $A$, the left-hand side of \eqref{eq-cossin-ineq-temp2} is a parabola with a downwards opening or a linear function. To show that \eqref{eq-cossin-ineq-temp2} holds for all $A\le A_0+\tfrac12\de$ with $1+A\ge\tfrac12\de$, it thus suffices to prove it for the two boundary values $-1+\tfrac12\de$ and $A_0+\tfrac12\de$ of $A$.

For $A=A_0+\tfrac12\de\ge 0$, we note that $1 + A \cos(\tfrac12\xi)^2 \phi(\xi)^2 \ge 1$ and
\[
  1 - A \sin(\tfrac12\xi)^2 \phi(\xi)^2 \ge 1 - A_0 \sin(\tfrac12\xi)^2 \phi(\xi)^2 - \tfrac12\de \ge \tfrac12\de
\]
by \eqref{eq-assumfilter-bounds} and \eqref{eq-assumfilter-largedata}, and hence \eqref{eq-cossin-ineq-temp2} holds for the right boundary value.

For $A=-1+\tfrac12\de\le 0$, we note that $1 - A \sin(\tfrac12\xi)^2\phi(\xi)^2 \ge 1$ and 
\[
1 + A \cos(\tfrac12\xi)^2 \phi(\xi)^2 \ge 1 + \klabig{-1+\tfrac12\de}  = \tfrac12\de 
\]
by \eqref{eq-assumfilter-bounds}, and hence \eqref{eq-cossin-ineq-temp2} also holds for the left boundary value.
\end{proof}

\begin{proof}[Proof of Lemma \ref{lem3}] 
We first observe that, by the Cauchy--Schwarz inequality, \eqref{eq-algebra} and \eqref{eq-assumfilter-bounds},
\[
\absBig{  \bigl\langle \mathcal{L}(\Phi u) v,v \bigr\rangle_{0}  -  \bigl\langle \mathcal{L}(u) v,v \bigr\rangle_{0} }  \le C_M \| \Phi u - u \|_1 \| v \|_0^2 \le C_{M} \tau \|v\|_{0}^2,
\]
where we use in the second inequality that $\abs{\phi(\xi)-1}\le   \min(2,c_0\xi^2) \le   C \xi$ for $\xi\ge 0$   by \eqref{propdephi1}.   
This yields
\begin{equation}\label{eq-inequ2}
\bigl\langle \mathcal{L}(\Phi u) v,v \bigr\rangle_{0} \geq \bigl\langle \mathcal{L}(u) v,v \bigr\rangle_{0} - C_M \tau \|v\|_{0}^2.
\end{equation}

Next, we use semiclassical pseudodifferential calculus as presented in Appendix \ref{appendix}. 
We first express the operator $\mathcal{L}(u)$ by quantization of certain symbols. We set, for $x\in\mathbb{T}$ and $\xi\in\R$ and with $\omega=\sqrt{\xi^2+ \tau^2}$,
\begin{align*}
 b_1(x,\xi) &= \phi(\omega), \quad  b_2(x,\xi) = \ka\, a(u(x)),\quad
 b_3(x,\xi) = \cos(\omega), \quad  b_4(x,\xi) = \sin(\omega)^2 \phi(\omega)^2.
\end{align*}
  As stated there is now somewhat unfortunately $\tau$ dependence in the $\omega$ symbol, however the dependence upon our semiclassical parameter $\tau$ arises as a very small, bounded perturbation and hence does not effect any of the semiclassical bounds used.  Note in addition that this slight notational complication arises from the generality of treating Klein-Gordon type operators with our semi-classical formulation and the $\tau$ dependence in $\omega$ would not arise in the wave equation setting.

With the corresponding quantizations $\Op_{b_1}^\tau,\dots,\Op_{b_4}^\tau$   (see equation \eqref{eq-Op} in Appendix \ref{appendix}),   we then have
\[
\mathcal{L}(u) = \Op_{b_1}^\tau \Op_{b_2}^\tau \Op_{b_3}^\tau \Op_{b_1}^\tau - \tfrac14 \Op_{b_1}^\tau \Op_{b_2}^\tau \Op_{b_4}^\tau \Op_{b_2}^\tau \Op_{b_1}^\tau.
\]
Note that all the symbols $b_1,\dots,b_4$ are in $S_{\si, 1}\cap S_{\si+1, 0}$ for $\si=1>\frac12$ since $u$ is in $H^2$   and $\phi$ is has bounded derivative,   and that we have
\[
\abs{b_j}_{\si,1} \le C_M, \qquad \abs{b_j}_{\si+1,0}\le C_M \myfor j=1,\dots,4
\]
by \eqref{eq-algebra-plus}. By \eqref{eq-algebra1}, this also holds for finite products of these symbols.   We refer to Appendix \ref{appendix} for the definition of the symbol classes $S_{\si,0}$ and $S_{\si+1,1}$ and the corresponding seminorms $\abs{\cdot}_{\si,1}$ and $\abs{\cdot}_{\si+1,0}$.   

By using Proposition \ref{propsemi}   repeatedly,   we thus have that 
\[
  \normbig{ \mathcal{L}(u)v - \Op_{b}^\tau(v) }_0 \le C_M \tau \norm{v}_0
\]
with the new symbol
\[
 b(x,\xi) = b_1(x,\xi)b_2(x,\xi)b_3(x,\xi)b_1(x,\xi) - \tfrac14 b_1(x,\xi)b_2(x,\xi)b_4(x,\xi)b_2(x,\xi)b_1(x,\xi).
\]
  This estimate and the Cauchy--Schwarz inequality imply   
  \begin{equation}
  \label{ineq3} \bigl\langle \mathcal{L}(u) v,v \bigr\rangle_{0}
   \geq  \bigl\langle  \Op_{b}^\tau v,v \bigr\rangle_{0} - C_{M} \tau \|v \|_{0}^2.
   \end{equation}

Next we use Proposition \ref{garding} to estimate the term with $\Op_{b}^\tau$ in \eqref{ineq3} further. Note that the symbol $b$ is in $S_{\si+1, 0}\cap S_{\si+1, 1}$ for $\si=1>\frac12$ since $u$ is in $H^2$   and $\phi$ has bounded derivative,   and that we have
\[
\abs{b}_{\si +1,0} \le C_M, \qquad \abs{b}_{\si+1,1}\le C_M.
\]
Note also that 
\[
1 + b (x, \xi)  \geq  \tfrac12 \delta  \qquad\text{for all}\qquad   x\in\mathbb{T},\, \xi\in\R
\]
by Lemma \ref{lem4} (with $A=\ka\, a(u(x))$).
By using Proposition \ref{garding}, we thus obtain that 
\[
\norm{v}_0^2 + \bigl\langle \Op_{b}^\tau v,v\bigr\rangle_{0} \ge \tfrac14 \delta \norm{v}_0^2 - C_M \tau \norm{v }_0^2.
\]
Combining this estimate with \eqref{eq-inequ2} and \eqref{ineq3} yields the statement of the lemma for sufficiently small $\tau$.
      \end{proof}
      
      \begin{proposition}\label{lowerBoundU}
Fix $M>0$ and $\delta>0$, and let the filter functions satisfy Assumptions \ref{assum_slw_psibds}--\ref{assum_lkap} with constants $0<\delta<1$ and $A_0\ge 0$.
Then,  there exists  $ \tau_{0}>0$ such that for every $\tau \le \tau_{0}$, for every $(e,\dot{e})\in H^2\times H^1$ and for every $u\in H^2$ with
\[
\norm{u}_2\le M,\qquad 1+  \kappa\, a(u(x)) \geq  \tfrac12 \delta > 0, \qquad   \kappa\, a(u(x)) \le A_0 + \tfrac12\delta, 
\]
we have that the modified energy \eqref{eq-energy} controls the Sobolev norms, i.e.,
\begin{equation}\label{modSob}
C_{\delta} \normv{(e,\dot{e})}_1^2 \leq \Ec(e,\dot{e},u)\leq C_M \normv{(e,\dot{e})}_1^2
\end{equation}
for two positive constants $C_{\delta}$, $C_M$.
\end{proposition}
\begin{proof} 
First note that for $\tau$ sufficiently small and $0 < \delta <1$ we have that
\begin{equation*}\label{essb:U}
\bigl(\tfrac18 \delta-1\bigr) \bigl\Vert \partial_x^2 e \bigr\Vert_0^2 \leq \ka\, \Uc \kla{\Phi e, \Phi u} \leq  C_{M}\Vert e\Vert_2^2,
\end{equation*}
where the upper bound follows from \eqref{eq-algebra} and the lower bound is a consequence of Lemma \ref{lem3} (with $v=\partial_x^2e$) using the representation of $\kappa\,\Uc(\Phi e, \Phi u)$ given in \eqref{repU}. 
 The bound \eqref{modSob} on the modified energy then follows by its definition in \eqref{eq-energy}.
      \end{proof}

\subsection{Local error bound}\label{subsec-locerror}

Similarly as in the semilinear case \cite{Gauckler2015}, but under higher regularity assumptions on the initial value, we can prove the following local error bound for the numerical method~\eqref{eq-method}. 

\begin{lemma}[Local error bound in $H^2\times H^1$]\label{lemma-localerror-small}
Let the filter functions satisfy Assumptions \ref{assum_slw_psibds} and \ref{assum_qlw_psi}. If $(u,\partial_t u)$ is a solution to \eqref{eq-qlw-compact} in $H^{5}\times H^4$ with
\[
\normvbig{\klabig{u(\cdot,t),\partial_t u(\cdot,t)}}_4\le M \myfor t_n\le t\le t_{n+1},
\]
then we have
\[
\normv{(d_{n+1},\dot{d}_{n+1})}_{1} \le C_M \tau^3 \abs{\ka} 
\]
for the local error $(d_{n+1},\dot{d}_{n+1})$ of \eqref{eq-localerror}.
\end{lemma}
\begin{proof}
Without loss of generality, we consider the case $n=0$, that is, we consider the local error
\[
(d_1,\dot{d}_1) = (u_1,\dot{u}_1) - \klabig{ u(\cdot,\tau),\partial_t u(\cdot,\tau) } . 
\]
As in the semilinear case \cite{Gauckler2015}, the proof relies on a comparison of the method in the form \eqref{eq-method-onestep} with the variation-of-constants formula for the exact solution $(u(\cdot,\tau),\partial_t u(\cdot,\tau))$. For initial values \eqref{eq-qlw-init} at time $0$, the va\-ri\-ation-of-con\-stants formula reads
\[
\begin{pmatrix} u(\cdot,\tau)\\ \partial_t u(\cdot,\tau) \end{pmatrix} = 
R(\tau) \begin{pmatrix} u_0\\ \dot{u}_0 \end{pmatrix}
+ \ka 
\int_0^\tau 
R(\tau-t) \begin{pmatrix} 0\\ f(u(\cdot,t)) \end{pmatrix}\,\drm t
\]
with
\[
R(t) = \begin{pmatrix} \cos(t\Om) & t \sinc(t\Om) \\ -\Om \sin(t\Om) & \cos(t\Om) \end{pmatrix}.
\]
Note that this formula makes sense in $H^{2}\times H^{1}$ for solutions in $H^{5}\times H^4$.
Using this formula, the local error is seen to be of the form
\begin{subequations}\label{eq-local}\begin{align}
&\begin{pmatrix} u_1-u(\cdot,\tau) \\ \dot{u}_1-\partial_t u(\cdot,\tau) \end{pmatrix}
= \tfrac12 \tau \ka R(\tau) \begin{pmatrix} 0\\ \widehat{f}(u_0)-f(u_0) \end{pmatrix}
+ \tfrac12 \tau \ka         \begin{pmatrix} 0\\ \widehat{f}(u_1)-f(u_1) \end{pmatrix}\label{eq-local1}\\
 &\qquad + \tfrac12 \tau \ka \begin{pmatrix} 0\\ f(u_1)-f(u(\cdot,\tau)) \end{pmatrix}\label{eq-local3}\\
 &\qquad + \tfrac12 \tau \ka \klabigg{ R(\tau) \begin{pmatrix} 0\\ f(u_0) \end{pmatrix} + R(0) \begin{pmatrix} 0\\ f(u(\cdot,\tau)) \end{pmatrix} }
- \ka \int_0^\tau R(\tau-t) \begin{pmatrix} 0\\ f(u(\cdot,t)) \end{pmatrix}\,\drm t
.\label{eq-local2}
\end{align}\end{subequations}
We estimate the three contributions \eqref{eq-local1}--\eqref{eq-local2} to the local error separately.

(a) The contributions to the local error in the first line \eqref{eq-local1} are due to the introduction of filters in the nonlinearity $\widehat{f}(u)=\Psi_1 f(\Phi u)$. Using that $R(\tau)$ preserves the norm $\normv{\cdot}_{1}$, we get
\[
\normvbigg{ R(\tau) \begin{pmatrix} 0\\ \widehat{f}(u_0)-f(u_0) \end{pmatrix} }_{1} = \norm{\widehat{f}(u_0)-f(u_0)}_{1}.
\]
We then split $\widehat{f}(u_0)-f(u_0)=\Psi_1(f(\Phi u_0)-f(u_0)) + (\Psi_1 f(u_0) - f(u_0))$ and use
\[
\normbig{\Psi_1\klabig{f(\Phi u)-f(u)}}_{1} \le C_{M} \norm{u}_{3} \norm{\Phi u - u}_{3} \le C_M \norm{u}_3 \norm{\tau^2\Omega^2u}_3 \le C_{M} \tau^2 \norm{u}_{5}^{2}
\]
by the Lipschitz property \eqref{eq-lipschitz} and the assumptions \eqref{eq-assumfilter-bounds} on the filter functions as well as 
\[
\norm{\Psi_1 f(u) - f(u)}_{1} \le C \tau^2 \norm{f(u)}_{3}\le C_{M} \tau^2\norm{u}_{5}^{2}
\]
by \eqref{eq-assumfilter-bounds} and \eqref{eq-algebra-f}. This shows that
\[
\normvbigg{ R(\tau) \begin{pmatrix} 0\\ \widehat{f}(u_0)-f(u_0) \end{pmatrix} }_{1} \le C_M \tau^2.
\]
The term in \eqref{eq-local1} with $u_1$ instead of $u_0$ can be dealt with in the same way using in addition that $\norm{u_1}_5\le C_M$ by Lemma~\ref{lemma-singlestep} since $\normv{(u_0,\dot{u}_0)}_4\le M$. This finally yields
\begin{equation}\label{eq-localerror-est1a}
\normv{ \text{term on right-hand side of \eqref{eq-local1}} }_{1} \le C_M \tau^3 \abs{\ka}.
\end{equation}
In the same way we also get
\begin{equation}\label{eq-localerror-est2a}
\normv{ \text{term on right-hand side of \eqref{eq-local1}} }_{2} \le C_M \tau^2 \abs{\ka},
\end{equation}
if we use $\abs{1-\psi_1(\xi)}\le C \xi$ and $\abs{1-\phi(\xi)}\le C \xi$ (which follow from \eqref{eq-assumfilter-bounds}) instead of $\abs{1-\psi_1(\xi)}\le C \xi^2$ and $\abs{1-\phi(\xi)}\le C \xi^2$.

(b) The contribution to the local error in the third line \eqref{eq-local2} is the quadrature error of the trapezoidal rule. With the corresponding second-order Peano kernel $K_2(\si)=\frac12 \si(\si-1)$, it takes the form
\[
\text{term \eqref{eq-local2}} = - \tau^3 \ka \int_0^1 K_2(\si)   h ''(\si  \tau ) \,\drm\si \qquad\text{with}\qquad   h (t) = R(\tau-t) \begin{pmatrix} 0\\ f(u(\cdot,t)) \end{pmatrix}.
\]
We thus have to estimate $\normv{h''(\si \tau)}_{1}$. For that we use, for $\ell=0,1,2$,
\[
\normvbigg{\frac{\drm^\ell}{\drm t^\ell} R(t) \begin{pmatrix} v\\ \dot v\end{pmatrix}}_{1} = \normv{(v,\dot v)}_{1+\ell} \myand 
\normbigg{\frac{\drm^{2-\ell}}{\drm t^{2-\ell}} f\klabig{u(\cdot,t)}}_{1+\ell} \le C_M
\]
by \eqref{eq-algebra} (and in the case $\ell=0$ also \eqref{eq-qlw-compact} to replace $\partial_t^2 u$) since $(u,\partial_t u)$ is bounded in $H^5\times H^4$. This yields
\begin{equation}\label{eq-localerror-est1b}
\normv{ \text{term \eqref{eq-local2}} }_{1} \le C_M \tau^3 \abs{\ka}.
\end{equation}
In the same way we also get
\begin{equation}\label{eq-localerror-est2b}
\normv{ \text{term \eqref{eq-local2}} }_{2} \le C_M \tau^2 \abs{\ka},
\end{equation}
if we use the first-order Peano kernel and $h'$ instead of the second-order Peano kernel and $h''$.

(c) The contribution to the local error in the second line \eqref{eq-local3} concerns only the error in the velocities. Using that we thus already have $\norm{u_1-u(\cdot,\tau)}_3\le C_M \tau^2 \abs{\ka}$ by the estimates \eqref{eq-localerror-est2a} and \eqref{eq-localerror-est2b} and that
\[
\normbig{ f(u_1) - f\klabig{u(\cdot,\tau)}}_1 \le C_M \norm{u_1-u(\cdot,\tau)}_3
\]
by \eqref{eq-lipschitz} and Lemma \ref{lemma-singlestep}, we get
\[
\normv{ \text{term \eqref{eq-local3}} }_{1} \le C_M \tau^3 \abs{\ka}.
\]
Together with the estimates \eqref{eq-localerror-est1a} and \eqref{eq-localerror-est1b} in (a) and (b), this completes the proof of the stated local error bound.
\end{proof}

In view of \eqref{eq-erroraccum}, we are not so much interested in local errors $(d,\dot{d})$ in the $H^2\times H^1$-norm as estimated in the previous lemma, but instead in energy differences of the form $\Ec\kla{e,\dot{e},u} 
- \Ec\kla{e-d,\dot{e}-\dot{d},u}$, where $(d,\dot{d})$ is a local error and $(e,\dot{e})$ is a global error. This extension is done in the following lemma.

\begin{lemma}\label{lemma-localerror-energy-small}
Let the filter functions satisfy Assumptions \ref{assum_slw_psibds} and \ref{assum_qlw_psi}. If $(e,\dot{e})\in H^2\times H^1$, $(d,\dot{d})\in H^2\times H^1$ and $u\in H^2$ with
\[
\norm{e}_2\le M, \qquad
\norm{d}_2\le M, \qquad
\norm{u}_2\le M,
\]
then we have 
\[
\absbig{\Ec\kla{e,\dot{e},u} 
- \Ec\kla{e-d,\dot{e}-\dot{d},u}}
\le C_M \klabig{ \tau^{-1} \abs{\ka}^{-1} \normv{(d,\dot{d})}_1^2 + \tau \abs{\ka}\, \normv{(e,\dot{e})-(d,\dot{d})}_1^2 } .
\]
\end{lemma}
\begin{proof}
The crucial property that we use in various forms is that, for any $\al>0$,
\begin{equation}\label{eq-cs}
2 \abs{ \skla{w,v} } \le \al^{-1} \norm{w}^2 + \al \norm{v}^2
\end{equation}
for a scalar product $\skla{\cdot,\cdot}$ with associated norm $\norm{\cdot}$, in particular
\[
\absbig{ \norm{v}^2 - \norm{v-w}^2} = \absbig{ \norm{w}^2 + 2 \skla{w,v-w} } \le (1+\al^{-1}) \norm{w}^2 + \al \norm{v-w}^2.
\]
This yields for the first term in the energy difference $\Ec\kla{e,\dot{e},u}- \Ec\kla{e-d,\dot{e}-\dot{d},u}$
\[
\absbig{\normv{(e,\dot{e})}_1^2 - \normv{(e-d,\dot{e}-\dot{d})}_1^2} \le (1+\al^{-1}) \normv{(d,\dot{d})}_1^2 +  \al \normv{(e-d,\dot{e}-\dot{d})}_1^2.
\]
The second term of the energy is $\kappa\,\Uc(\widehat{e},\widehat{u})$ with $\widehat{u}=\Phi u$ and $\widehat{e}=\Phi e$. For the second term in $\Uc$, we get in the energy difference, with $\widehat{u}=\Phi u$, $\widehat{e}=\Phi e$ and $\widehat{d}=\Phi d$,
\begin{multline*}
\tfrac14 \tau^2 \abs{\ka} \absBig{ \normbig{\Psi_1 \klabig{a(\widehat{u})\partial_x^2\widehat{e}} }_1^2 - \normbig{\Psi_1 \klabig{a(\widehat{u})\partial_x^2(\widehat{e}-\widehat{d})} }_1^2}\\
  \le \tfrac14 \tau^2 \abs{\ka} (1+\al^{-1}) \normbig{\Psi_1 \klabig{a(\widehat{u})\partial_x^2\widehat{d}\,} }_1^2
 + \tfrac14 \tau^2 \abs{\kappa} \, \al \, \normbig{\Psi_1 \klabig{a(\widehat{u})\partial_x^2(\widehat{e}-\widehat{d})} }_1^2,
\end{multline*}
and hence, by \eqref{eq-assumfilter-bounds}, \eqref{eq-assumfilter-psi1phi} and \eqref{eq-lipschitz} with $s=0$ (similarly as in \eqref{eq-est-aux}), we get the bound 
$C_M (1+\al^{-1}) \norm{\widehat{d}}_2^2 + C_M \al \norm{\widehat{e}-\widehat{d}}_2^2$
for this term.
For the first term in $\Uc$, we have 
\begin{align*}
&  \sklabig{ \cos(\tau\Om) \widehat{e}, a(\widehat{u})\partial_x^2\widehat{e} }_1  - \sklabig{ \cos(\tau\Om) (\widehat{e}-\widehat{d}), a(\widehat{u})\partial_x^2(\widehat{e}-\widehat{d}) }_1 \\
&\,  = \sklabig{\cos(\tau\Om)\widehat{d},a(\widehat{u})\partial_x^2\widehat{e}}_1
   + \sklabig{\cos(\tau\Om) (\widehat{e}-\widehat{d}), a(\widehat{u})\partial_x^2\widehat{d}}_1\\
&\,  = \sklabig{\cos(\tau\Om)\widehat{d},a(\widehat{u})\partial_x^2\widehat{d}}_1 
 + \sklabig{\cos(\tau\Om)\widehat{d},a(\widehat{u})\partial_x^2(\widehat{e}-\widehat{d}) }_1 + \sklabig{\cos(\tau\Om) (\widehat{e}-\widehat{d}), a(\widehat{u})\partial_x^2\widehat{d}}_1 .
\end{align*}
Using partial integration, \eqref{eq-lipschitz} and \eqref{eq-cs}, we get for these terms similarly as above the bound $C_M (1+\al^{-1}) \norm{\widehat{d}}_2^2 + C_M \al \norm{\widehat{e}-\widehat{d}}_2^2$. By choosing $\al=\tau \abs{\ka}$, the statement of the lemma then follows from assumption \eqref{eq-assumfilter-bounds}.
\end{proof}

From Lemmas \ref{lemma-localerror-small} and \ref{lemma-localerror-energy-small}, we get the following bound for the local error in the form as it appears in \eqref{eq-erroraccum}.

\begin{proposition}[Local error bound in the energy]\label{prop-localerror}
Let the filter functions satisfy Assumptions \ref{assum_slw_psibds} and \ref{assum_qlw_psi}. If $(u,\partial_t u)$ is a solution to \eqref{eq-qlw-compact}  in $H^{5}\times H^4$ with
\[
\normvbig{\klabig{u(\cdot,t),\partial_t u(\cdot,t)}}_4\le M \myfor t_n\le t\le t_{n+1},
\]
and if $(u_{n},\dot{u}_{n})\in H^2\times H^1$ is a corresponding numerical solution with
\[
\normv{(u_{n},\dot{u}_{n})}_1 \le 2 M,
\]
then we have 
\begin{multline*}
\absbig{\Ec\kla{e_{n+1},\dot{e}_{n+1},u_{n+1}} 
- \Ec\kla{e_{n+1}-d_{n+1},\dot{e}_{n+1}-\dot{d}_{n+1},u_{n+1}}} \\
\le C_M \tau^5 \abs{\ka} + C_M \tau \abs{\ka} \, \normv{(e_n,\dot{e}_n)}_1^2
\end{multline*}
with the global errors $(e_{n},\dot{e}_{n})$ and $(e_{n+1},\dot{e}_{n+1})$ of \eqref{eq-globalerror} and the local error $(d_{n+1},\dot{d}_{n+1})$ of \eqref{eq-localerror}. 
\end{proposition}
\begin{proof}
We apply Lemma \ref{lemma-localerror-energy-small} with $(e,\dot{e})=(e_{n+1},\dot{e}_{n+1})$, $(d,\dot{d})=(d_{n+1},\dot{d}_{n+1})$ and $u=u_{n+1}$ in combination with Lemma \ref{lemma-localerror-small}. Note that $\norm{d_{n+1}}_2\le C_M$ by Lemma \ref{lemma-localerror-small}, $\norm{u_{n+1}}_2\le C_M$ by Lemma \ref{lemma-singlestep} and $\norm{e_{n+1}}_2\le C_M$ by the assumption on the exact solution and by the just mentioned bound on $u_{n+1}$. Lemmas \ref{lemma-localerror-small} and \ref{lemma-localerror-energy-small} then yield the stated estimate but with $(e_{n+1},\dot{e}_{n+1})-(d_{n+1},\dot{d}_{n+1})$ instead of $(e_n,\dot{e}_n)$ on the right-hand side. To get the final statement, we use the definitions \eqref{eq-globalerror} of $(e_{n+1},\dot{e}_{n+1})$ and \eqref{eq-localerror} of $(d_{n+1},\dot{d}_{n+1})$ and apply Lemma \ref{lemma-stability-naive}.
\end{proof}

\begin{remark}
  The local error bound of Lemma \ref{lemma-localerror-small} is only based on the estimates \eqref{eq-algebra-f} and \eqref{eq-lipschitz}. As they extend to nonzero $g$ in $f$ and \eqref{eq-qlw}, also Lemma \ref{lemma-localerror-small} and Proposition \ref{prop-localerror} extend to this case.
\end{remark}

\subsection{Error accumulation and proof of Theorem \ref{thm-large}}\label{subsec-accumulation-small}

We proceed as outlined in Section \ref{subsec-outline}. Considering the numerical solution $(u_n,\dot{u}_n)$ given by \eqref{eq-method} for $n=0,1,\ldots$, we set
\[
\Ec_n(e,\dot{e}) = \Ec(e,\dot{e},u_n).
\]
We then decompose $\Ec_{n+1}(e_{n+1},\dot{e}_{n+1})$ with the global error $(e_{n+1},\dot{e}_{n+1})$ of \eqref{eq-globalerror} as in \eqref{eq-erroraccum}.

  Adapting the usual inductive argument to prove global error bounds, we assume that the numerical solution $(u_j,\dot{u}_j)$ satisfies, for $j=0,\dots,n$ and with the constants $M$, $A_0$ and $\de$ of Assumption \ref{assum-exact},
\begin{subequations}\label{eq-induction}
\begin{align}
  &\normv{(u_j,\dot{u}_j)}_1 \le 2 M\label{eq-induction-bound}
\end{align}
and
\begin{align}
&1 + \kappa \, a(u_j) \ge \tfrac12 \de \myand
\ka\,a(u_j) \le A_0 + \tfrac12 \de.\label{eq-induction-de-phi}
\end{align}
\end{subequations}
This is clear for $j=0$ by Assumption \ref{assum-exact}. Under these hypotheses, we will prove the error bound of Theorem~\ref{thm-large} until time $t_n=n\tau$. We will then prove that \eqref{eq-induction} also holds for $j=n+1$ to close the inductive argument.

Under the regularity assumption \eqref{eq-regularity} on the exact solution and thanks to \eqref{eq-induction-bound},
we get from Propositions \ref{prop-stability} and \ref{prop-localerror} 
\[
\absbig{\Ec_{j+1}(e_{j+1},\dot{e}_{j+1})} \le \absbig{\Ec_{j}(e_{j},\dot{e}_{j})} + C_M \tau \abs{\ka} \,\normv{(e_j,\dot{e}_j)}_1^2 + C_M \tau^5 \abs{\ka}
\]
for $j=0,\dots,n$.
Thanks to \eqref{eq-induction-de-phi}, we can then apply Proposition \ref{lowerBoundU} (with $u=u_j$) to get
\[
\absbig{\Ec_{j+1}(e_{j+1},\dot{e}_{j+1})} \le \klabig{1 + C_{M} \tau \abs{\ka}} \absbig{\Ec_{j}(e_{j},\dot{e}_{j})} + C_M \tau^5 \abs{\ka}
\]
for $j=0,\dots,n$. Solving this recursion in the standard way yields the error bound
\begin{equation}\label{eq-boundcalE-temp}
\absbig{\Ec_{j+1}(e_{j+1},\dot{e}_{j+1})} \le C_{M} \tau^4 \e^{C_{M} \abs{\ka} t_{j+1} } 
\end{equation}
for $j=0,\dots,n$. By applying once again Proposition \ref{lowerBoundU}, we get the global error bound
\begin{equation}\label{eq-bounde-temp}
\normv{(e_{j+1},\dot{e}_{j+1})}_1^2 \le C_M \tau^4 \e^{C_M \abs{\ka} t_{j+1} },
\end{equation}
for $j=0,\dots,n-1$.

In order to close the induction, we have to justify that \eqref{eq-induction} also holds for $j=n+1$. 
To do so, we note that the bound on a single time step given in Lemma \ref{lemma-stability-naive}, the local error bound of Lemma \ref{lemma-localerror-small} and the bound \eqref{eq-bounde-temp} for $j=n-1$, allow us to estimate\footnote{Note that this estimate is not a proof of \eqref{eq-bounde-temp} for $j=n$, because then the constants would explode. Instead, it is only used to justify \eqref{eq-induction} for $j=n+1$.}
\[
\normv{(e_{n+1},\dot{e}_{n+1})}_1 \le C_M \normv{(e_{n},\dot{e}_{n})}_1 + \normv{(d_{n},\dot{d}_{n})}_1 \le C_{M,t_{n+1}} \tau^2. 
\]
This implies $\norm{u(\cdot,t_{n+1})-u_{n+1}}_2 \le C_{M,t_{n+1}} \tau^2$, and hence \eqref{eq-induction-bound} also holds for $j=n+1$ by assumption \eqref{eq-regularity} provided that $\tau$ is sufficiently small.
It also implies $\norm{u(\cdot,t_{n+1})-u_{n+1} }_{L^\infty} \le C_{M,t_{n+1}} \tau^2$, 
and hence, again for sufficiently small $\tau$,  
\[
\Vert \kappa \, a(u(\cdot,t_{n+1})) - \kappa\, a(u_{n+1}) \Vert_{L^\infty} \le \tfrac12 \de.
\]
By assumptions \eqref{eq-delta} and \eqref{eq-A0}, this shows that \eqref{eq-induction-de-phi} also holds for $j=n+1$. This closes the induction and concludes the proof of Theorem \ref{thm-large}.

\section{Proof of the error bound for the full discretization}\label{sec-error-full}

In this section, we study fully discrete methods \eqref{eq-method-full}. We first give the proof of Theorem \ref{thm-full} on their global error. The structure of this proof is the same as for the semi-discretization in time in Section \ref{sec-error}. We study stability in Section \ref{subsec-stability-full} below, control Sobolev norms with the energy in Section \ref{subsec-sobolevnorm-full}, estimate the local error Section \ref{subsec-locerror-full} and put everything together in Section \ref{subsec-proof-full}. All arguments are extensions to the fully discrete setting of the arguments of Section \ref{sec-error}, which illustrates the importance of proving such semi-discrete error bounds first.

Throughout, we use, for $s\ge s'\ge 0$, the approximation property
\begin{equation}\label{eq-Pc-error}
\normbig{v-\mathcal{P}^K(v)}_{s'} \le K^{-(s-s')} \norm{v}_{s} \myfor v\in H^{s}
\end{equation}
of the $L^2$-orthogonal projection $\Pc^K$ of \eqref{eq-Pc}, and its stability 
\begin{equation}\label{eq-Pc-stability}
\normbig{\mathcal{P}^K(v)}_{s} \le \norm{v}_{s} \myfor v\in H^{s}.
\end{equation}
 In addition, we use, for $s\ge s'\ge 0$ with $s-s'>\frac12$, the approximation property
\begin{equation}\label{eq-Ic-error}
\normbig{v-\mathcal{I}^K(v)}_{s'} \le   C_{s,s'}   K^{-(s-s')} \norm{v}_{s} \myfor v\in H^{s}
\end{equation}
of the trigonometric interpolation $\mathcal{I}^K$, and its stability
\begin{equation}\label{eq-Ic-stability}
\normbig{\mathcal{I}^K(v)}_{s} \le   C_s   \norm{v}_{s} \myfor v\in H^{s}.
\end{equation}

We emphasize that all estimates in the following are uniform in the spatial discretization parameter $K$.

\subsection{Stability of the numerical method}\label{subsec-stability-full}

Our aim is to show that the stability estimates of Section \ref{subsec-stability-small} carry over to the fully discrete situation. 

Starting with the definition of the energy $\Ec$ of \eqref{eq-energy}, we define its fully discrete version 
\begin{equation}\label{emod-full}
\Ec^K(e, \dot e,u) = \normv{(e,\dot{e})}_1^2 + \kappa \,\Uc^K \kla{\Phi e,\Phi u}
\end{equation}
with
\begin{equation}\label{Uc-full}
\Uc^K \kla{ e, u}
= \sklabig{ \cos(\tau\Om) \partial_{x}^2 e,  \mathcal{P}^K \klabig{a^K(u) \partial_{x}^2 e} }_0 - \tfrac{1}{4}\tau^2 \kappa \normbig{ \Psi_1 \mathcal{P}^K \klabig{a^K(u) \partial_{x}^2 e} }_1^2.
\end{equation}
The difference compared to the $\Ec$ of \eqref{eq-energy} are the additional projections $\mathcal{P}^K$ and the functions $a^K=\mathcal{I}^K\circ a$ instead of $a$. 

The computation of Lemma \ref{lemma-energy} directly transfer to the new energy \eqref{emod-full} of the fully discrete setting if we use that $\Psi_1$ and $\mathcal{P}^K$ commute and if we replace the function $f$ by $\mathcal{P}^K\circ f^K$, where $f^K$ is defined in \eqref{fK}. 

In order to transfer the bound of the remainder term of Lemma \ref{lemma-remainder-energyest} to the fully discrete setting, we use the bounds \eqref{eq-Pc-stability} and \eqref{eq-Ic-stability}  on $\Pc^K$ and $\mathcal{I}^K$ (to estimate $a^K=\mathcal{I}^K\circ a$, which appears in $f^K$) and in addition the property
\begin{equation}\label{eq-property-Pc}
\sklabig{ v^K , \mathcal{P}^K(w) }_s = \sklabig{ v^K , w }_s \qquad\text{for}\qquad v^K\in\mathcal{V}^K, \, w \in H^s
\end{equation}
with $s=1$.
This property is needed for the symmetry argument  and the partial integrations in  the proof of Lemma \ref{lemma-remainder-energyest}.\footnote{It is not immediately clear, whether these steps can also be done if the (traditional) trigonometric interpolation is used instead of the projection $\mathcal{P}^K$ to define $\widehat{f}^K$  in the fully discrete method.}

From the fully discrete versions of Lemmas \ref{lemma-energy} and \ref{lemma-remainder-energyest}, we finally get the stability estimate of Proposition \ref{prop-stability} also in the fully discrete setting for $\Ec^K$.

\subsection{Controlling Sobolev norms with the energy}\label{subsec-sobolevnorm-full}

We show that the bounds on the energy $\Ec$ of Proposition \ref{lowerBoundU} carry over to the fully discrete setting and the corresponding energy $\Ec^K$ of \eqref{emod-full}. 

For the upper bound in \eqref{modSob}, we proceed as in the proof of Proposition \ref{lowerBoundU} and use in addition \eqref{eq-Pc-stability} and \eqref{eq-Ic-stability} to deal with the additional $\mathcal{P}^K$ and $\mathcal{I}^K$ (in $a^K=\mathcal{I}^K\circ a$). 

For the lower bound in \eqref{modSob}, we use that we have by \eqref{eq-assumfilter-psi1phi}, \eqref{eq-Pc-stability} and \eqref{eq-property-Pc} with $s=0$ 
\begin{align*}
\kappa \, \Uc^K \kla{ \Phi e, \Phi u} &\ge \kappa \sklabig{ \cos(\tau\Om) \partial_{x}^2 \Phi e,  a^K(\Phi u) \partial_{x}^2 \Phi e }_0 - \tfrac{1}{4}\tau^2 \kappa^2 \normbig{ \Psi_1 \klabig{a^K(\Phi u) \partial_{x}^2 \Phi e} }_1^2\\
 &= \sklabig{\mathcal{L}^K(\Phi u) \partial_x^2 e , \partial_x^2 e }_0
\end{align*}
for a trigonometric polynomial $e$
(note that, in comparison to \eqref{Uc-full}, the projections $\mathcal{P}^K$ are absent on the right) with 
\[
\mathcal{L}^K(u) = \ka \Phi a^K(u)\cos(\tau\Om) \Phi - \tfrac{1}{4} \kappa^2 \Phi a^K(u) \sin^2(\tau\Om) \Phi^2 a^K(u) \Phi. 
\]
The operator $\mathcal{L}^K$ is the same as the operator $\mathcal{L}$ of Section \ref{subsec-sobolevnorm-full}, except that $a$ is replaced by $a^K=\mathcal{I}^K\circ a$. To obtain the lower bound in \eqref{modSob}, we can thus proceed exactly as in the proof of Proposition \ref{lowerBoundU} if we replace $a$ by $a^K$ in the statement of this proposition and restrict to trigonometric polynomials $e$.

\subsection{Local error bound}\label{subsec-locerror-full}

The main difference compared to the semi-discrete setting arises in the local error bound of Section \ref{subsec-locerror}, which now has to take also the spatial error into account. We denote here and in the following by
\begin{equation}\label{eq-uK}
u^K(\cdot,t) = \Pc^K\klabig{u(\cdot,t)}, \qquad \partial_t u^K(\cdot,t) = \Pc^K\klabig{\partial_t u(\cdot,t)}
\end{equation}
the $L^2$-orthogonal projection of the exact solution onto $\mathcal{V}^K$.

\begin{lemma}[Local error bound in $H^2\times H^1$]\label{lemma-localerror-full}
Let $s\ge 0$, and let the filter functions satisfy Assumptions \ref{assum_slw_psibds} and \ref{assum_qlw_psi}. If $(u,\partial_t u)$ is a solution to \eqref{eq-qlw} in $H^{5+s}\times H^{4+s}$ with
\[
\normvbig{\klabig{u(\cdot,t),\partial_t u(\cdot,t)}}_{4+s}\le M \myfor t_n\le t\le t_{n+1},
\]
then we have 
\[
\normvbig{\klabig{d_{n+1}^K,\dot{d}_{n+1}^K}}_{1}
 \le C_M \tau^3 \abs{\ka} + C_M \tau K^{-2-s} \abs{\ka}
\]
for the fully discrete local error
\[
\klabig{d_{n+1}^K,\dot{d}_{n+1}^K}=\ph_\tau^K\klabig{u^K(\cdot,t_n),\partial_t u^K(\cdot,t_n)}
 - \klabig{u^K(\cdot,t_{n+1}),\partial_t u^K(\cdot,t_{n+1})}
\]
with the fully discrete numerical flow $\ph_{\tau}^K$.
\end{lemma}
\begin{proof}
The proof is similar to the one of Lemma \ref{lemma-localerror-small}. We restrict again to the case $n=0$. Writing ${\tilde f}^K=\Pc^K\circ f$, we start from the fully discrete analog
\begin{align*}
&\begin{pmatrix} u_1^K-u^K(\cdot,\tau) \\ \dot{u}_1^K-\partial_t u^K(\cdot,\tau) \end{pmatrix}
= \tfrac12 \tau \ka R(\tau) \begin{pmatrix} 0\\ \widehat{f}^K (u_0)- {\tilde f}^K (u_0) \end{pmatrix}
+ \tfrac12 \tau \ka         \begin{pmatrix} 0\\  \widehat{f}^K (u_1^K)- {\tilde f}^K (u_1^K) \end{pmatrix}\\
 &\quad + \tfrac12 \tau \ka R(\tau) \begin{pmatrix} 0\\ \widehat{f}^K (u_0^K)- \widehat{f}^K (u_0) \end{pmatrix} + \tfrac12 \tau \ka \begin{pmatrix} 0\\  {\tilde f}^K (u_1^K)-{\tilde f}^K (u(\cdot,\tau)) \end{pmatrix}\\
 &\quad + \tfrac12 \tau \ka \klabigg{ R(\tau) \begin{pmatrix} 0\\  {\tilde f}^K (u_0) \end{pmatrix} + R(0) \begin{pmatrix} 0\\  {\tilde f}^K  (u(\cdot,\tau)) \end{pmatrix} }
- \ka \int_0^\tau R(\tau-t) \begin{pmatrix} 0\\  {\tilde f}^K (u(\cdot,t)) \end{pmatrix}\,\drm t
\end{align*}
of the local error representation \eqref{eq-local}. In the derivation of this representation, we have used that $\Pc^K$ and the four components of $R$ commute. The contributions to the local error can be estimated similarly as in the proof of Lemma \ref{lemma-localerror-small} using in addition the properties  \eqref{eq-Pc-error}--\eqref{eq-Ic-stability} of $\Pc^K$ and $\mathcal{I}^K$  and the assumed regularity of the exact solution.

(a) In the terms of the first line, we decompose $\widehat{f}^K - \tilde{f}^K = (\widehat{f}^K -\Pc^K\circ f^K) + \Pc^K\circ(f^K-f)$. For the terms with $\widehat{f}^K -\Pc^K\circ f^K$ (which correspond to \eqref{eq-local1}), we get an estimate $C_M \tau^3 \abs{\kappa}$ in $H^2\times H^1$ and $C_M \tau^2 \abs{\kappa}$ in $H^3\times H^2$ as in the proof of Lemma \ref{lemma-localerror-small}. For the terms with $\Pc^K\circ(f^K-f)$, we use in particular \eqref{eq-Ic-error} in addition to the arguments of the proof of Lemma \ref{lemma-localerror-small} to get an estimate $C_M \tau K^{-4-s} \abs{\kappa}$ in $H^2\times H^1$ and $C_M \tau K^{-3-s} \abs{\kappa}$ in $H^3\times H^2$.

(b) For the terms in the third line (which correspond to \eqref{eq-local2}), we get as in the proof of Lemma \ref{lemma-localerror-small}  an estimate $C_M \tau^3 \abs{\kappa}$ in $H^2\times H^1$ and $C_M \tau^2 \abs{\kappa}$ in $H^3\times H^2$.

(c) For the new first term in the second line, we get  an estimate $C_M \tau K^{-2-s} \abs{\kappa}$ in $H^2\times H^1$ and, using the properties \eqref{eq-assumfilter-bounds} and \eqref{eq-assumfilter-psi1phi} of the filters as in \eqref{eq-est-aux}, an estimate $C_M K^{-2-s} \abs{\kappa}$ in $H^3\times H^2$. The second term in the second line corresponds to \eqref{eq-local3} and can then be dealt with as in the proof of Lemma \ref{lemma-localerror-small}.
\end{proof}

Based on Lemma \ref{lemma-localerror-full}, we can then prove a corresponding local error bound in the energy as in Proposition \ref{prop-localerror}.

\subsection{Proof of Theorem \ref{thm-full}}\label{subsec-proof-full}

The error accumulation is done as in Section \ref{subsec-accumulation-small}, but with the exact solution replaced by its projection \eqref{eq-uK}. Note that we need \eqref{eq-induction-de-phi} for $a^K=\mathcal{I}^K\circ a$ instead of $a$; we use \eqref{eq-Pc-error} and \eqref{eq-Ic-error} to deal with that. This gives a the claimed global error bound of Theorem \ref{thm-full}, but with $u$ replaced by $u^K=\mathcal{P}^K\circ u$ in the error estimate. We then use once more \eqref{eq-Pc-error} to get the precise error estimate of Theorem \ref{thm-full}.

\begin{appendix}

\section{Semiclassical pseudodifferential calculus}\label{appendix}
 
In this section we shall recall the basic results about pseudodifferential calculus that were needed in our proof. The presentation follows closely \cite[Section 8]{HKR}.
We shall use the Fourier transform on the torus  defined by
\[
\hat{u}_j= (\mathcal{F}_{x}u)(j) = \frac{1}{2\pi} \int_{-\pi}^{\pi} u(x) \e^{-\iu jx} \, \drm x.
\]

We consider symbols $a(x, \xi)$  (not to be confused with the function $a$ in \eqref{eq-qlw})  defined on  $\mathbb{T} \times \mathbb{R} $ that are continuous in $\xi$, and we use the quantization
\begin{equation}\label{eq-Op}
(\Op_{a} u)(x)= \sum_{j \in \mathbb{Z}} a(x, j) \hat{u}_j \e^{\iu j x}.
\end{equation}
We introduce for $\si\ge 0$ the following seminorms of symbols:
\begin{align*}
   &  |a|_{\si,0}=   \sup_{| \alpha | \leq  \si } \bigl\| \mathcal{F}_{x}(\partial_{x}^\alpha a) \bigr\|_{L^2_{ j}(\mathbb{Z}, L^\infty_{\xi}(\R))}, \qquad  |a|_{\si,1}=  \sup_{| \alpha | \leq  \si } \bigl\|   \mathcal{F}_{x} (\partial_{x}^\alpha \partial_{\xi} a)\bigr\|_{L^2_{ j}(\mathbb{Z},L^\infty_{\xi}(\R))}.
\end{align*}
Note that
\[
\normbig{ \mathcal{F}_{x}(\partial_{x}^\alpha a) }_{L^2_{ j}(\mathbb{Z}, L^\infty_{\xi}(\R))}^2 = \sum_{j\in\Z} \normbig{ (\iu j)^\alpha \hat{a}_j }_{L^\infty(\R)}^2
\]
with the Fourier coefficients
\begin{equation}\label{eq-fourier-a}
\hat{a}_j(\xi)=(\mathcal{F}_{x}a)(j,\xi), \qquad j\in\Z, \, \xi\in\R,
\end{equation}
of $a$, and similarly
\[
\normbig{ \mathcal{F}_{x}(\partial_{x}^\alpha \partial_\xi a) }_{L^2_{ j}(\mathbb{Z}, L^\infty_{\xi}(\R))}^2 = \sum_{j\in\Z} \normbig{ (\iu j)^\alpha \hat{a}_j' }_{L^\infty(\R)}^2
\]
with $\hat{a}_j'=\frac{\drm}{\drm\xi}\hat{a}_j=\widehat{(\partial_\xi a)}_j=(\mathcal{F}_{x}\partial_\xi a)(j,\xi)$.
We shall say that $a\in S_{\si,0}$ if $|a|_{\si,0}<\infty$ and $a \in S_{\si,1}$ if  $ |a|_{\si,1}<\infty.$
The use of these seminorms compared to some more classical ones allows us to avoid to lose too many derivatives
while keeping very simple proofs.
Note that we can  easily relate  $|a|_{\si,0}$ to  more classical symbol  seminorms up to losing more derivatives.
For example, we have for every $\si \geq 0$
\[
\sup_{| \alpha | \leq  \si } \sup_{x\in\mathbb{T},  \,  \xi\in\R} |\partial_x^\alpha a(x, \xi)| \le C | a|_{\si+ s, 0}
\]
with $s>\frac12$.    The lower bound of $\frac12$ is related to the Sobolev embedding $H^s\hookrightarrow L^\infty$ in $1d$ and should be generalized to $\frac{d}{2}$ for higher dimensional generalizations of our arguments here.  

Writing a symbol $a(x,\xi)$ as a Fourier series with respect to its first variable $x$,
\[
a(x,\xi) = \sum_{j\in\Z} \hat{a}_j(\xi) \e^{\iu jx}
\]
with Fourier coefficients \eqref{eq-fourier-a}, the quantization \eqref{eq-Op} takes (formally) the form
\begin{equation}\label{eq-Op-fourier}
(\Op_{a} u) (x) = \sum_{l\in\Z} \klabigg{\sum_{k\in\Z} \hat{a}_{l-k}(k) \hat{u}_k} \e^{\iu lx}.
\end{equation}
Its $L^2(\mathbb{T})$-norm (which we denote in this appendix by $\norm{\cdot}_{L^2(\mathbb{T})}$ to avoid confusion with the other norms appearing here) is given by
\[
\norm{\Op_{a} u}_{L^2(\mathbb{T})}^2 = \sum_{l\in\Z} \absbigg{ \sum_{k\in\Z} \hat{a}_{l-k}(k) \hat{u}_k}^2 \le \sum_{l\in\Z} \absbigg{ \sum_{k\in\Z} \norm{\hat{a}_{l-k}}_{L^\infty(\R)}  \abs{\hat{u}_k} }^2.
\]
Noting that the upper bound on the right is the squared $L^2(\mathbb{T})$-norm of the product of the functions with Fourier coefficients $\norm{\hat{a}_{k}}_{L^\infty(\R)}$ and $\abs{\hat{u}_k}$, we get from \eqref{eq-algebra1} the following $L^2$ continuity result.

     \begin{proposition}
     \label{propcontinu}
      Assume that $\si> \frac12$. 
      Then, there exists $C>0$ such that for every  $a \in S_{\si, 0}$ and for every $u \in L^2(\mathbb{T})$, we have $\Op_{a} u\in L^2(\mathbb{T})$ with
\[
\|\Op_{a} u\|_{L^2(\mathbb{T})} \leq C |a|_{\si,0} \|u\|_{L^2( \mathbb{T})}.
\]
     \end{proposition}

For a very similar result, we refer to \cite[Proposition 8.1]{HKR} which slightly refines in terms of the regularity of the symbols, the classical results of $L^2$
continuity for symbols in $S^0_{0,0}$ that are compactly supported in $x$, see for example \cite{Taylor}. 

    We shall now state results of symbolic calculus, see also \cite[Proposition 8.2]{HKR}.
  \begin{proposition}\label{propcalculus}
Assume that $\si>\frac12$. Then, there exists $C>0$ such that for every $a \in S_{\si+1, 1}$,  we  have
    \begin{equation}
    \label{adj} \bigl\| (\Op_{a})^* (u) - \Op_{\overline{a}} (u)\bigr\|_{L^2(\mathbb{T})} \leq C |a|_{\si+1, 1} |u|_{L^2(\mathbb{T})}, 
     \end{equation}
     where $ (\Op_{a})^*$ is the adjoint of the operator $\Op_{a}$ for the $L^2(\mathbb{T})$ scalar product.
Moreover, we have for every $ a \in S_{\si,1}$ and $b\in S_{\si+1, 0}$ that $ab\in S_{\si,0}$ and 
   \begin{equation}
   \label{comp}
    \bigl\|\Op_{a} \Op_{b} (u) - \Op_{ab} (u)\bigr\| _{L^2(\mathbb{T})} \leq C |a|_{\si, 1} |b|_{\si+1, 0}\, \|u \|_{L^2( \mathbb{T})}.
    \end{equation}
  \end{proposition}

  \begin{proof}
  Let us first  prove \eqref{adj}. We start by computing a symbol $c$ with
\[
\Op_a^*= \Op_c.
\]
Using that, by \eqref{eq-Op-fourier},
\[
\sum_{l\in\Z} \overline{\hat{a}_{l-j}(j)} \hat{u}_l = \sklabig{\Op_a \e^{\iu jx}, u}_{L^2(\mathbb{T})} = \sklabig{\e^{\iu jx},\Op_c u}_{L^2(\mathbb{T})} = 
\sum_{l\in\Z} \hat{c}_{j-l}(l) \hat{u}_l ,
\]
we define such a symbol $c$ by
\[
\hat{c}_j(\xi) = \overline{\hat{a}_{-j}(\xi+j)}, \qquad j\in\Z, \, \xi\in\R.
\]
Assuming that $a \in S_{\si,0}$, we thus also have that $c \in S_{\si,0}$ with $|c|_{\si, 0} = | a|_{\si,0}$.
     Next, by Taylor expansion, we can write
\[
d(x,\xi):= c(x,\xi)- \overline{a(x,\xi)}=  \sum_{j\in\Z} \int_{0}^1 j \, \overline{\hat{a}_{-j}'(\xi+s j)}\, \drm s \,  \e^{\iu j x}.
\]
    We shall prove that $|d|_{\si,0}\leq | a|_{\si+1, 1}$ and the result will follow from Proposition \ref{propcontinu}. This estimate follows from
   \[
(\iu j)^{\alpha} \hat{d}_j(\xi) = (\iu j)^{\alpha} \int_{0}^1 j \, \overline{\hat{a}_{-j}'(\xi+s j)}\, \drm s = -\iu \int_{0}^1 \overline{(- \iu j)^{\alpha+1} \hat{a}_{-j}'(\xi+s j)}\, \drm s,
\]
which implies $\norm{(\iu j)^{\alpha} \hat{d}_j}_{L^\infty(\R)} \le \norm{(- \iu j)^{\alpha+1} \hat{a}_{-j}'}_{L^\infty(\R)}$,
and ends the proof of \eqref{adj}.
   
Let us now prove \eqref{comp}.  We first observe that for $\si>\frac12$ and $a \in S_{\si, 0}$, $b\in S_{\si,0}$, we have by \eqref{eq-algebra1} (applied with functions $u$ and $v$ with Fourier coefficients $\hat{u}_j = \|\hat{a}_j\|_{L^\infty(\R)}$ and $\hat{v}_j = \|\hat{b}_j\|_{L^\infty(\R)}$) that 
\begin{equation}\label{eq-symbol-algebra}
| ab |_{\si,0} \le C |a|_{\si, 0}|b|_{\si, 0}, 
\end{equation}
 and thus that $ ab \in S_{\si,0}$.

Next, we compute a symbol $e$ with
\[
\Op_{a} \Op_{b} = \Op_{e}.
\]
Such a symbol is obtained by writing $\Op_b$ as in \eqref{eq-Op-fourier} and $\Op_a$ as in \eqref{eq-Op}, which yields
\[
e(x, \xi)= \sum_{j \in \mathbb{Z}} a(x, \xi + j) \hat{b}_j(\xi) \e^{\iu j x}.
\]
We then get that
\[
f(x,\xi) := e(x, \xi) - a(x, \xi) b(x, \xi)=  \sum_{j \in \mathbb{Z}} \int_{0}^1 \partial_{\xi} a(x,  \xi+ s j)\, \drm s\, j \, \hat{b}_j(\xi) \e^{\iu j  x} .
\]
We next estimate a suitable seminorm of the symbol $f$ to apply Proposition \ref{propcontinu}.
           By taking the Fourier transform in $x$, we obtain that
\[
 (\iu l)^\alpha \hat{f}_l(\xi) = (\iu l)^\alpha \sum_{k\in \mathbb{Z}} \int_{0}^1\hat{a}_{l-k}'(\xi+ sk)\, \drm s \,k \, \hat{b}_k(\xi).
\]
We thus have for $\alpha\ge 0$
\[
\normbig{ (\iu l)^{\alpha} \hat{f}_l}_{L^\infty(\R)} 
  \le \abs{l}^{\alpha} \sum_{k\in \mathbb{Z}} \normbig{ \hat{a}_{l-k}'}_{L^\infty(\R)} \normbig{ (\iu k) \hat{b}_k}_{L^\infty(\R)} ,
\]
from which we obtain for $\si>\frac12$ by \eqref{eq-algebra1} that
\[
|f|_{\si,0} \le C | a|_{\si, 1} |b|_{\si+1, 0}.
\]
Since by definition of $f$, we have $ \Op_{a} \Op_b - \Op_{ab}= \Op_{f}$, the result follows from Proposition \ref{propcontinu}.
\end{proof}
     
  We shall next  define a semiclassical version of the above calculus which is the one of interest for us.
     For any symbol $a(x, \xi)$ as above,  we set for $0<\tau\le 1$  
\[
a^\tau(x, \xi)= a(x, \tau \xi)
\]
and we define 
      \begin{equation*}
 \Op^{\tau}_{a} u =
  \Op_{a^\tau} u.
 \end{equation*}
For this calculus, we have the following result, see also \cite[Proposition 8.3]{HKR}.

 \begin{proposition}
 \label{propsemi}
Assume that $\si> \frac12$. 
Then, there exists $C>0$ such that for every $0< \tau \le 1$,  we have 
  \begin{itemize}
  \item 
    for every  $a \in S_{\si,0}$
\[
\bigl\|\Op_{a}^{\tau} u\bigr\|_{L^2( \mathbb{T})} \leq C |a|_{\si,0} \|u\|_{L^2( \mathbb{T})}, 
\]
   \item for every $a \in  S_{\si, 1}$ and for every $b \in S_{\si+1, 0}$
\[
\bigl\|\Op_{a}^{\tau} \Op_{b}^{\tau} (u) - \Op_{ab}^{\tau} (u)\bigr\|_{L^2(\mathbb{T})} \leq C \tau  |a|_{\si, 1} |b|_{\si+1, 0}\, \|u \|_{L^2( \mathbb{T})}
\]
  \item for every $a \in S_{\si+1,1}$ 
\[
\bigl\|(\Op_{a}^{\tau})^* (u) - \Op_{\overline{a}}^{\tau} (u)\bigr\|_{L^2(\mathbb{T})} \leq C \tau  |a|_{\si+1, 1}\, \|u \|_{L^2( \mathbb{T})}.
\]
  
  \end{itemize}
 
 \end{proposition}
 
 \begin{proof}
  The results are direct consequences of Propositions \ref{propcontinu} and \ref{propcalculus}
   since for any symbol $a$, we have by definition of $a^\tau$  that $|a^\tau|_{\si, 0}= |a|_{\si,0}$ and $|a^\tau |_{\si, 1}= \tau |a|_{\si,1}$.
 \end{proof}
  
 Let us finally state the semiclassical G{\aa}rding inequality.
  \begin{proposition}
 \label{garding}
Assume that $\si> \frac12$. 
For $a \in S_{\si  +1,0}\cap S_{\si+1,1}$ assume further that there exists $\delta>0$
           such that
\[
a(x, \xi) \geq  \delta \myforall x\in\mathbb{T}, \, \xi\in\R.
\]
           Then, there exists $C>0$ which depends only on $|a|_{\si +1, 0}$, $|a|_{\si+1, 1}$ and $\delta$ such that
\[
\langle \Op_{a}^\tau u, u \rangle_{L^2(\mathbb{T})}
          \geq \tfrac12 \delta \|u\|_{L^2(\mathbb{T})}^2 - C \tau \|u \|_{L^2(\mathbb{T})}^2 \myforall 0<\tau\le 1.
\]
   \end{proposition}     
     \begin{proof}
      We can write that
\[
a(x, \xi) = \tfrac12 \delta + b(x,\xi)^2, \qquad b(x,\xi)= \bigl( a(x,\xi)- \tfrac12 \delta \bigr)^{1/2}.
\]
 We will show below that,  since $a \geq \delta >0$, we also have that $b \in S_{\si+1,0} \cap S_{\si+1, 1}$  with
\begin{equation}\label{eq-estb}
\abs{b}_{\si+1,0}\le C \myand \abs{b}_{\si+1,1}\le C,
\end{equation}
where $C$ depends only on $|a|_{\si +1, 0}$, $|a|_{\si+1, 1}$ and $\delta$. 
       By using Proposition \ref{propsemi}, we thus get that
\[
\Op_{a}^\tau= \tfrac12 \delta + (\Op_{b}^\tau)^* \Op_{b}^\tau + R^\tau
\]
        with
\[
\|R^\tau u\|_{L^2(\mathbb{T})} \leq C \tau \|u\|_{L^2(\mathbb{T})}.
\]
      The result follows easily.

 It remains to show \eqref{eq-estb}. We restrict here to $\si=1$, which is the value of $\si$ that is needed in Section \ref{sec-error}. The proof for other  values of $\si$ is similar, but with longer formulas. In the following, we write
\[
F(y) = \bigl( y- \tfrac12 \delta \bigr)^{1/2},
\]
such that $b(x,\xi)=F(a(x,\xi))$ and we observe that $F$ is a smooth function on $[\delta, + \infty[$. 

For the first estimate of \eqref{eq-estb}, we start from
\begin{equation*}\label{eq-garding-start}
\abs{b}_{2,0} \le \abs{F(a)}_{0,0} + \absbig{\partial_x^2 F(a)}_{0,0} \le \abs{F(a)}_{0,0} + \absbig{F'(a)\partial_x^2 a}_{0,0} +  \absbig{F''(a)(\partial_x a)^2}_{0,0}.
\end{equation*}
To estimate the products further, we use the estimate $\abs{cd}_{0,0} \le \abs{c}_{0,0} \abs{d}_{1,0}$, which follows from \eqref{eq-algebra1} in the same way as \eqref{eq-symbol-algebra}. This yields
\begin{equation}\label{eq-garding-step1}
\abs{b}_{2,0} \le \abs{F(a)}_{0,0} + \abs{F'(a)}_{1,0} \abs{\partial_x^2 a}_{0,0} +  \abs{F''(a)}_{1,0} \abs{\partial_x a}_{1,0} \abs{\partial_x a}_{0,0}.
\end{equation}
To finish the proof, we just need to explain how to estimate $| G(a)|_{1,0}$ for some smooth $G$ which is smooth
 on the image of $a$.  We start from 
\begin{multline*}
\abs{G(a)}_{1,0} \le \normbig{\mathcal{F}_x \klabig{(1+\partial_x)G(a)}}_{L^2_j(\Z,L_\xi^\infty(\R))}
 \le C \normbig{\mathcal{F}_x \klabig{(1-\partial_x^2) G(a)}}_{L_j^\infty(\Z,L_\xi^\infty(\R))}\\
 \le C \normbig{\mathcal{F}_x \klabig{(1-\partial_x^2) G(a)}}_{L_\xi^\infty(\R,L^2_j(\Z))}
 = C \norm{G(a)}_{L^\infty_\xi(\R,H^2_x(\mathbb{T}))},
\end{multline*}
where the second estimate follows from  
\begin{align*}
 \sum_{j\in\mathbb{Z}} \normbig{ \klabig{\mathcal{F}_x (1+\partial_x)G(a)}(j)}_{L^\infty_\xi(\R)}^2 &= \sum_{j\in\mathbb{Z}} \frac{1+j^2}{(1+j^2)^2} \normbig{ \klabig{\mathcal{F}_x (1-\partial_x^2)G(a)}(j)}_{L^\infty_\xi(\R)}^2\\
&\le \biggl( \sum_{j\in\mathbb{Z}} \frac{1}{1+j^2}\biggr) \sup_{j\in\mathbb{Z}} \normbig{ \klabig{\mathcal{F}_x (1-\partial_x^2)G(a)}(j)}_{L^\infty_\xi(\R)}^2, 
\end{align*}
  and the third estimate follows from interchanging the two $L^\infty$-norms and estimating the $L_j^\infty(\Z)$-norm by the $L_j^2(\Z)$-norm. 

Next, we can use \eqref{eq-algebra-plus} to get that for every $\xi$,
\[ \|G(a) \|_{H^2_{x}(\mathbb{T})} \leq\Lambda(\norm{a}_{H_x^2(\mathbb{T})}))( 1 + \|a\|_{H^2_{x}(\mathbb{T})})\]
where $\Lambda(\cdot)$ stands again for a continuous non-decreasing function that can change from line to line as a stand in for the dependence upon the algebra of the calculus established here. Therefore we finally obtain that
\[ \norm{G(a)}_{L^\infty_\xi(\R,H^2_x(\mathbb{T}))} \leq\Lambda(\norm{a}_{L^\infty_\xi(\R,H^2_x(\mathbb{T}))})(1 +  \|a\|_{L^\infty_\xi(\R,H^2_x(\mathbb{T}))}) 
 \leq  \Lambda(|a|_{2,0} ) . 
\]
 Using this estimate and the above estimate of $\abs{G(a)}_{1,0}$ in \eqref{eq-garding-step1} completes the proof of the first estimate of \eqref{eq-estb}.  The proof of the second estimate of \eqref{eq-estb} is very similar. 
 \end{proof}

\end{appendix}

\subsection*{Acknowledgement}

  We thank the referees for their valuable suggestions and comments.   
This work was supported by Deutsche Forschungsgemeinschaft (DFG) through project GA 2073/2-1 (LG), SFB 1114 (LG) and SFB 1173 (KS) and by National Science Foundation (NSF) under grants DMS-1454939 (JL), DMS-1312874 (JLM) and DMS-1352353 (JLM).

{\small 

}

\end{document}